\documentclass[10pt,reqno]{amsart}
\usepackage{amsmath,amsfonts,amssymb,amscd,amsthm,amsbsy,bbm, epsf,calc,enumerate,color,graphicx,verbatim,cite,import}
\usepackage{color}
\usepackage{datetime,soul}

\usepackage{ifpdf}
\ifpdf
    \usepackage[colorlinks = false]{hyperref}
\fi

\newcounter{hours}\newcounter{minutes}

\makeatletter
\newtheorem*{rep@theorem}{\rep@title}
\newcommand{\newreptheorem}[2]{%
\newenvironment{rep#1}[1]{%
 \def\rep@title{#2 \ref{##1}}%
 \begin{rep@theorem}}%
 {\end{rep@theorem}}}
\makeatother

\theoremstyle{plain}
\newtheorem{definition}{Theorem}[section]
\newtheorem{thm}{Theorem}[section]
\newreptheorem{thm}{Theorem}

\newtheorem{lem}[thm]{Lemma}
\newreptheorem{lem}{Lemma}
\newtheorem{cor}[thm]{Corollary}
\newtheorem{prop}[thm]{Proposition}

\theoremstyle{definition}

\newtheorem{DEF}[definition]{Definition}

\newtheorem{property}{Property}

\theoremstyle{remark}                  

\def\E{{\mathbb E}}
\def\F{{\mathcal F}}

\def\P{{\mathbb P}}
\def\R{{\mathbb R}}

\def\real{{\mathbb R}}

\def\indicator{{\mathbf 1}}
\def\integer{{\mathbb Z}}

\def\ep{\varepsilon}
\def\e{\varepsilon}

\def\diam{\textnormal{diam}}

\def\II{I\hspace{-2pt}I}
\def\IV{I\hspace{-2pt}V}
\def\III{I\hspace{-2pt}I\hspace{-2pt}I}

\definecolor{darkgreen}{rgb}{0,0.4,0}

\def\liminf{\mathop{\lim\,\inf}\limits}%
\def\limsup{\mathop{\lim\,\sup}\limits}%
\DeclareMathOperator*{\esssup}{ess\,sup}
\numberwithin{equation}{section}

\begin{document}

\title{Homogenization and Non-Homogenization of certain Non-Convex Hamilton-Jacobi Equations}
\author{William M. Feldman \and Panagiotis E. Souganidis }
\thanks{W. M. Feldman partially supported by NSF-RTG  grant DMS-1246999 .}
\thanks{P. E. Souganidis was partially supported by the NSF grants  DMS-1266383  and   DMS-1600129   .}
\email{feldman@math.uchicago.edu}
\email{souganidis@math.uchicago.edu}
\address{Department of Mathematics, The University of Chicago, Chicago, IL 60637, USA}
\maketitle

\begin{abstract}
We continue the study of the homogenization of coercive non-convex Hamilton-Jacobi equations in random media identifying two  general classes of  Hamiltonians with very distinct behavior. For the first class there is no  homogenization in a particular environment while for the second homogenization takes place   in environments with finite range dependence. Motivated by the recent counter-example of Ziliotto \cite{Ziliotto}, who constructed  a coercive but non-convex Hamilton-Jacobi equation with stationary ergodic random potential field for which homogenization does not hold, we show that same happens for coercive Hamiltonians which have  a strict saddle-point, a very local property.  We also identify, based on  the recent  work of Armstrong and Cardaliaguet \cite{ArmstrongCardaliaguet} on the homogenization of positively homogeneous random Hamiltonians in environments with finite range dependence, a new general class Hamiltonians, namely equations with uniformly strictly star-shaped sub-level sets,
which homogenize.

\end{abstract}

\section{Introduction}
We continue the study of the homogenization of non-convex Hamilton-Jacobi equations in random media identifying two  general classes of  Hamiltonians with very distinct averaging behavior. 
\medskip

In particular we consider the asymptotic behavior, as $\e \to 0$, of the initial value problem 
\begin{equation}\label{eqn: HJ1}
u^\e_t + H(Du^\e,\tfrac{x}{\e}) = 0  \ \hbox{ in } \  \real^d \times (0,\infty) \qquad  u(x,0) = u_0(x),
\end{equation}
with  $H(\xi,x)$ a stationary ergodic random field in $x\in \real^d$ which is uniformly coercive in $\xi$ but not convex.  

\medskip

A simple rescaling shows that the limiting behavior of the $u^\ep$'s with initial datum $u(x,0)=\xi \cdot x$ for some $\xi \in \R^d$ is the same as the long time average behavior of the solution $u$ of the initial value problem
\begin{equation}\label{takis0}
u_t + H(Du, x) = 0  \ \hbox{ in } \  \real^d \times (0,\infty) \qquad  u(x,0) =\xi \cdot x. 
\end{equation}
Indeed the following is true for all $t>0$:
 \[ \liminf_{\ep  \to 0} u^\ep (0,t)=  \liminf_{s \to \infty} \frac{u(0,s)}{s} \ \text{and}  \ \limsup_{\ep \to 0} u^\ep (0,t)= \limsup_{s \to \infty} \frac{u(0,s)}{s}.\]
There is large body of work about the qualitative and quantitative homogenization of convex Hamilton-Jacobi equations in random media following respectively the papers of Souganidis \cite{Sou}, Rezakhanlou and Tarver \cite{RT}, Lions and Souganidis \cite{LionsSouganidis1,LionsSouganidis2}, Armstrong and Souganidis \cite{ArmstrongSouganidis,ArmstrongSouganidis2},   and Armstrong, Cardaliaguet and Souganidis \cite{ArmstrongCardaliaguetSouganidis}.

\medskip

In spite of this progress whether non-convex coercive Hamilton-Jacobi equations homogenize or not in random environments has been a long  standing open question with some  results available for Hamiltonians of very specific form; see Armstrong, Tran and Yu \cite{ATY,ATY2} and Gao \cite{Gao}. 

\medskip

A significant step in the understanding of the problem is the recent work of  Ziliotto \cite{Ziliotto}, who constructed a  non-convex Hamiltonian $h=h(\xi)$ and a stationary ergodic random field $V=V(x)$, such that \eqref{eqn: HJ1} with Hamiltonian $H(\xi,x) = h(\xi) - V(x)$ does not homogenize for a particular linear initial datum.  

\medskip

In the direction of homogenization the most recent progress is due to Armstrong and Cardaliaguet \cite{ArmstrongCardaliaguet}, who considered, among others,  \eqref{takis0} in environments with finite range of dependence assumption in $x$ and Hamiltonians $H(\xi,x)$ which are positively-homogeneous in the gradient variable $\xi$.

\medskip

Our work builds upon these two contributions. In the non-homogenization direction we identify a general class of coercive nonlinearities $h$, namely  ones with a strict  saddle-point, which, when used in the separated Hamiltonian $H(\xi,x) = h(\xi) - V(x)$ with a potential  $V$ similar to the one in \cite{Ziliotto}, give rise  to non-homogenization. This result shows that non-homogenization may be a consequence of a local property of the nonlinearity. The outcome is also consistent with a conjecture by the second author that homogenization for  non-convex Hamiltonians, in general, may fail,  if the media obstructs the ``characteristics'' from reaching in time $t$ all directions  at distance $\text{O}(t)$.
 
\medskip

In the direction of homogenization we are able to adapt the arguments  of  \cite{ArmstrongCardaliaguet} to the class of Hamiltonians with star-shaped sub-level sets, which  strictly includes  Hamiltonians which are positively homogeneous in the gradient variable.

\medskip

To state the result we first need make explicit the meaning of the strict-saddle point, which, of course, makes sense only for $d\geq 2$. 

\medskip

We say that a continuous function $h:\R^d \to \R$ has a strict saddle point at some $\xi_0 \in \R^d$, if there exist two non-trivial orthogonal subspaces $P_1$ and $P_2,$ which together span $\real^d$, and $\eta >0$ so that for any unit vectors $e \in P_1$ and $f \in P_2$,
\begin{equation}\label{saddle}
 h(\xi_0 + \lambda e) < h(\xi_0) \ \hbox{ for } \ \lambda \in [-\eta,\eta] \setminus \{0\} \ \hbox{ and } \ h(\xi_0 + \lambda f) > h(\xi_0) \ \hbox{ for } \  \lambda \in [-\eta,\eta]\setminus \{0\}.
 \end{equation}
This is, for example, the case when $h$ is $C^2$ at $\xi_0$ and has a strict saddle point in the usual calculus sense, that is 
\begin{equation}\label{saddle smooth}
\begin{cases}
 Dh(\xi_0) = 0  \text{ and there exist $m \in \{1, \ldots,  d-1\}$ such that}\\[1mm]
 D^2h(\xi_0) \ \hbox{ has $m$ strictly positive eigenvalues and $d-m$ strictly negative eigenvalues.}
 \end{cases}
\end{equation}
Of course \eqref{saddle} contains more general non-smooth examples, like the Hamiltonian of \cite{Ziliotto}, and smooth examples like $h(\xi)= \xi_1^2 - \xi_2^4$ which has a strict saddle at the origin in $\real^2$. 

\medskip

Our first result is:
\begin{thm}\label{thm: intro non-hom}
Let $h : \real^d \to \real$ be a continuous coercive Hamiltonian with a strict saddle point at some $\xi_0 \in \R^d$. There exist  a stationary ergodic random field $V$ such  that  the solution $u^\ep$ of \eqref{eqn: HJ1} with  Hamiltonian $H(\xi,x) = h(\xi) - V(x)$ and initial datum $u_0(x) = \xi_0 \cdot x$ does not homogenize in the sense that, for all $t>0$, 
\[ \liminf_{\ep  \to 0} u^\ep (0,t) < \limsup_{\ep \to 0} u^\ep (0,t) \ \hbox{ almost surely.} \]
\end{thm}

Actually we are able to relax the assumption that $ Dh(\xi_0) = 0$ in \eqref{saddle},  if we allow for space-time stationary ergodic Hamiltonians. The result is:

\begin{cor}\label{cor: intro non-hom space-time}
Let $h : \real^d \to \real$ be a smooth, coercive with super-linear growth Hamiltonian and suppose that there is  $\xi_0 \in \real^d$ such that, for some $1 \leq m \leq d-1$, $ D^2h(\xi_0)$  has $m$ strictly positive  and $d-m$ strictly negative eigenvalues.  There exists a space-time stationary ergodic random field $V=V(x,t)$ such that the solution $u^\ep$ to the initial value problem 
 $$ u^\ep_t + h(Du^\ep) - V(\tfrac{x}{\e},\tfrac{t}{\e}) = 0 \ \hbox{ in } \ \real^d \times (0,\infty) \ \qquad  u(x,0) = \xi_0 \cdot x$$ 
 does not homogenize, in the sense that, for all $t>0$, 
$$ \liminf_{\ep  \to 0} u^\ep (0,t) < \limsup_{\ep \to 0} u^\ep (0,t) \ \hbox{ almost surely.} $$
\end{cor}

\medskip

Considering the homogenization result we show that the essential feature of the Hamiltonians used in \cite{ArmstrongCardaliaguet} is not positive homogeneity in the gradient variable but rather star-shapedness of the sub-level sets. 

\medskip

A Hamiltonian $H$ is said to have star-shaped sub-level sets (with respect to the origin) if  
\begin{equation}\label{eqn: intro star-shapedness}
\text{for every $\mu \geq \inf_{\real^d} H(\cdot,x)$ and $x \in \real^d$,  $\{\xi: H(\xi,x) \leq \mu\}$ is strictly star-shaped with respect to $0$.}
\end{equation}
An easy generalization of our methods allows to consider Hamiltonians with each $\mu$ sub-level strictly star-shaped with respect to a different point $\xi_\mu$ (independent of $x$).  Actually, since the method is quantitative in nature, it is necessary  to quantify the strict star-shapedness of the sub-level sets; more precise details can be found in section~\ref{sec: hamiltonian assumptions quant}.  
\medskip

The homogenization result we are obtaining here is an extension of \cite{ArmstrongCardaliaguet} which studied positively homogeneous Hamiltonians  and of  \cite{ArmstrongSouganidis} which studied quasi-convex Hamiltonians, that is $H$'s with convex sub-level sets. 

\medskip

A key step in establishing homogenization is to understand the limiting behavior of the approximate correctors $v^\delta(\cdot; \xi)$ for $\xi \in \real^d$, that is  the  solutions of the problem,
\[\delta v^\delta + H(\xi + Dv^\delta,x) = 0 \ \hbox{ in } \ \real^d.\]
In view of the assumed coercivity, homogenization follows (see, for example, \cite{LionsSouganidis1,LionsSouganidis2,ArmstrongSouganidis}) if it is shown that 
\begin{equation}\label{eqn: a cor limit}
-\delta v^\delta(0,\xi) \to \overline{H}(\xi) \ \hbox{ almost surely for every $\xi \in \real^d$.}
\end{equation}
 
\medskip

We state our hommogenization result as a rate of convergence for \eqref{eqn: a cor limit};  the statement here is rather vague and is weaker than the result we actually prove since we are avoiding stating the technical assumptions in the introduction.  The full result can be found in Proposition~\ref{prop: a correctors full rate}.
\begin{thm}\label{thm: intro hom}
If $H$ is a finite range of dependence stationary random field of Hamiltonians satisfying standard coercivity and continuity assumptions and (a quantitative version of) \eqref{eqn: intro star-shapedness}, then, there exists a modulus $\rho$ and constants $C=C(H(\xi))>0$ and $c=c(H(\xi))>0$ such that, for every $\delta, \lambda>0$, 
$$\P(|\delta v^\delta(0,\xi)+\overline{H}(\xi)| >\lambda \delta \rho (\delta|\log\delta|)) \leq C\exp(-c\lambda^{4\wedge d}).$$ 
\end{thm}

The approach in this paper, as in \cite{ArmstrongSouganidis,ArmstrongCardaliaguetSouganidis,ArmstrongCardaliaguet}, is not to study the approximate corrector problem directly but instead to focus on the so called metric problems associated with the Hamiltonian. 
 
\medskip

In particular, given a plane with unit direction $e$,  we study the asymptotic properties of the non-negative  solution $m_\mu$  of the metric problem to the plane
\begin{equation}\label{eqn: intro metric}
H(Dm_\mu,x) = \mu \ \hbox{ in } \ \{ x \cdot e >0\} \ \hbox{ with } \ m_\mu = 0 \ \hbox{ on } \ \mathcal{H}_e:=\{ y\in \R^d: y \cdot e  = 0\}.
\end{equation}
In the context of level set motions this would perhaps be better called the arrival time problem since, in that case,  the solution $m_\mu(x)$ corresponds to the time that a front, which  started from the plane $ \mathcal{H}_e,$ reaches the point $x$.  Then the level sets of $\{y\in \R^d: m_\mu(y) = t\}$ correspond to the locations of the propagating front at times $t>0$.  At first glance this interpretation seems to be limited to positively $1$-homogeneous level set motion type Hamiltonians. However, we show here that metric problems for Hamiltonians with star-shaped sub-levels can be transformed to metric problems for level-set motion type equations.

\medskip

Although we are doubtful that the star-shaped assumption is necessary for homogenization, we explain next why it is (in a sense) necessary for a method like that of \cite{ArmstrongCardaliaguet}, which is based on the solutions of the planar metric problem.  

\medskip

The star-shapedness assumption guarantees uniqueness for metric problems like \eqref{eqn: intro metric}. As a matter of fact, for spatially homogenous Hamiltonians, it is indeed the sharp assumption for uniqueness.  With $x$ dependence,  it is not clear what is the correct assumption for uniqueness of the planar metric problems.  

\medskip

In fact uniqueness for the associated metric problems is not necessary to obtain a homogenized Hamiltonian. For example, \cite{ATY}  deals with this non-uniqueness of the metric problem solutions for  Hamiltonians of the form
$$H(\xi,x) = \Psi(h(\xi)) - V(x),$$
with $h\geq 0$ convex and coercive and $\Psi: [0,\infty) \to \real$ smooth such that  $\lim_{s\to \infty}\Psi(s)=\infty$. Speaking very informally in \cite{ATY}  the metric  problem  $ \Psi(h(Du)) - V(x) = \mu$ is replaced by $h(Du) = \Psi^{-1}(\mu + V(x)),$ where $\Psi^{-1}$ is multi-valued. It turns out, however, that in each continuous branch of $\Psi^{-1}$, the metric problems  for $h$ have unique solutions.  

\medskip

Conceivably there might exist some more general method of parametrizing the non-unique solutions of the planar metric problem as solutions of sub-problems with a more amenable Hamiltonian, however this seems quite difficult to realize.

\medskip

We remark that, in view of our conclusions, the homogenization results in  \cite{ATY} are rather unstable. For example, it is easy to see that arbitrary small perturbations of $H(p,x)=(|p|^2-1)^2 - V(x)$, a typical problem studied in   \cite{ATY}, have strict-saddle points and, hence, homogenization fails. 

\subsection*{Organization of the paper} The paper is divided into two parts dealing separately with the non-homogenization and homogenization results.  Section~\ref{sec: non-hom}, which is  about  non-homogenization, is subdivided itself into several subsections. First we give the set up and precise assumptions on the Hamiltonian and the random field (subsection~\ref{sec: assumptions1}), then in subection~\ref{sec: non-hom 2} we construct barriers which imply non-homogenization, subsections~\ref{sec: Constructing the random field} and~\ref{sec: Properties of the random field} are about the construction of the random field and its properties. Subsection~\ref{sec: mixing prop} is about a mixing property satisfied by the field, and, finally, subsection~\ref{sec: space-time} is about the space-time result. Section 3 is dedicated to the homogenization result and is also divided in several subsections. In subsection~\ref{sec: hom outline} we give a general outline of the proof of quantitative homogenization developed in \cite{ArmstrongCardaliaguetSouganidis,ArmstrongCardaliaguet}. In subsection~\ref{sec: assumptions}  we give the specific assumptions on the Hamiltonian, explain more rigorously the role of star-shapedness, and discuss  possible generalizations. In subsection~\ref{sec: reduction} we explain the reduction to the $1-$positively homogeneous metric problem, and in subsection~\ref{sec:planar} we introduce the planar problem.  In subsections~\ref{sec: The fluctuations estimate}-\ref{sec: bias estimate} we go through the quantitative homogenization proof in more detail and,  in particular, the fluctuations estimate and the convergence of the expectations (bias estimate), giving detailed proofs in the places where the star-shapedness condition comes in.  The effective equation is introduced in subsection~\ref{sec:effective} where we also describe some of its properties, and, finally, in subsection~\ref{sec:connections} we touch upon the relationship with the approximate problem 

\subsection*{Notation} We work in $\R^d$ with the Eucledian metric $d$, and   $S^{d-1}$ and $B_r(x)$ are respectively the unit sphere and the open ball in $\R^d$ centered at $x$ with radius $r$ and $B_r:=B_r(0).$ For $x\in \R^d$, $|x|_\infty:=\max_{i \in \{1,\ldots, m\}}|x_i|$.  Given $a, b\in \R$, $a \vee b:=\max(a,b)$ and $a\wedge b:=\min (a,b)$. We write $A^C$ for the complement of the set $A$. For $A, B \subset \R^d$ compact, we  denote by $d(A,B)$ is their Hausdorff distance.  If  $\mathcal A$ is a collection of subsets of a certain set $\Omega$, 
$\sigma(\mathcal A)$ is the smallest $\sigma$-algebra on $\Omega$ containing $\mathcal A$. Given two quantities  $A$ and $B$, $A \lesssim B$ means that there exists some $C>0$ depending only in the data such  $A \leq CB$. If there $C_1B\leq A \leq C_2B$ we write $A  \approx B$. Finally, $[x]$ denotes the integer part of $x$, and, for $p\in \R^d\setminus \{0\}$, $ \hat p:p/|p|$. 

\subsection*{Terminology} A constant $C$ is said to depend on the data and we write $C(\text{data}),$ if it depends only the several constants in the assumptions. 
Throughout the paper, unless otherwise said, solutions should be interpreted as the Crandall-Lions viscosity solutions. 

\subsection*{Acknowledgments}  The authors would like to thank Scott Armstrong for pointing out the connection between our mixing result Lemma~\ref{lem: mix est} and the notion of polynomial mixing.
%
\section {Non-homogenization for Hamiltonians with a strict saddle point}\label{sec: non-hom}
\subsection{Assumptions and results}\label{sec: assumptions1}
We consider Hamiltonians of the form
\begin{equation}
H(\xi,x) = h(\xi) - V(x).  
\end{equation}
In addition to \eqref{saddle} 
we assume that   
\begin{equation}\label{takis-1}
\text{ $h$ is coercive,}
\end{equation}
and 
\begin{equation}\label{takis1}
\hbox{ $V: \R^d \to [-1,1]$ is  $2$-Lipschitz  stationary ergodic random field on $\real^d$.}
\end{equation}
Although we are yet not introducing here the probability space, we will freely write $\omega \in \Omega$.
\medskip

We show  that the potential $V$ constructed in \cite{Ziliotto} can be used to prove  non-homogenization for this more general class of Hamiltonians;  recall that   \cite{Ziliotto} works with a very specific $h$.
\medskip

We point out that, if \eqref{saddle} holds for some $\eta$, then it holds  for all  smaller $\eta$. Thus the non-homogenization behavior is due to the random field and the local behavior near a generic saddle point of the Hamiltonian and not the global behavior of the specific $h$ used in \cite{Ziliotto}.

\medskip

Since the arguments for $d\geq 3$ are similar to the ones when $d=2$, here, to simplify the presentation,  we assume that $d=2$.

\medskip

Then \eqref{saddle} has the following simpler form:   There exist  two orthogonal unit directions $e_1$ and $e_2$ and $\eta >0$ such  that, for  $ s \in [-\eta,\eta] \setminus \{0\}$, 
\begin{equation}\label{saddle2}
h(\xi_0 + s e_1) < h(\xi_0) <  h(\xi_0 +  s e_2). 
\end{equation}

To further simplify the presentation and the arguments, below we make some more  deductions from assumption \eqref{saddle2}.  After  a rotation, we can assume that $e_1$ and $e_2$ are the axis directions.  Adding a constant and using $\lambda^{-1}h(\xi_0 + \eta^{-1}\xi)$  in place of $h$,  we can assume that $h(\xi_0) = 0$, $\xi_0 = 0$ and  $\eta = 1$.  Finally, by selecting  $\lambda$ appropriately, we have 
$$0<3 \leq \min \{ -h(\pm e_1), h(\pm e_2)\}.$$
The continuity of $h$ yields that, for some small $\delta >0$ and $c_\delta>0$,  
\begin{equation}\label{eqn: h in square}
\begin{cases}
 h(\xi) \leq -2 \ \hbox{ in } \ B_\delta(\pm e_1),  \ \ h(\xi ) \geq 2 \ \hbox{ in } \ B_\delta(\pm e_2), \ \text{and}\\[1.5mm]
 h(\xi) \leq c_\delta \ \hbox{ in } \ |\xi_2| \leq \delta, \  |\xi_1| \leq 1 \ \hbox{ and } \ h(\xi ) \geq -c_\delta \ \hbox{ in } \ |\xi_1| \leq \delta, \ |\xi_2| \leq 1.
 \end{cases}
 \end{equation}
We will assume that we have a random field $V$ which looks -- vis-\`{a}-vis the Hamiltonian $h$ -- qualitatively different at infinitely many different dyadic time/length scales $T_n = 2^n$.

\medskip

More precisely, we assume that $V$ has the following two properties:
\begin{property}\label{prop: horizontal segment}
For every $\theta >0$ and  almost surely, there exists a sequence $n_k(\omega) \to \infty$ such that,
for all $k \geq 1$,  $V(x) = -1$ on a horizontal line segment of length $\frac{4}{\delta}T_{n_k}$ centered at distance from the origin smaller than $[\theta  T_{n_k}]$.
\end{property}

\begin{property}\label{prop: vertical segment}
For every $\theta >0$  and almost surely, there exists a sequence $n_k'(\omega) \to \infty$ such that, for all $k \geq 1$, $V(x) = 1$ on  a vertical line segment of length $\frac{4}{\delta}T_{n_k'}$ centered  at distance smaller than $[\theta T_{n_k'}]$.
\end{property}

In subsection~\ref{sec: Constructing the random field} and ~\ref{sec: Properties of the random field} we construct a random field satisfying \eqref{takis1} and Properties~\ref{prop: horizontal segment} and \ref{prop: vertical segment}.  For the motivation, coming from the differential game interpretation of the Hamilton-Jacobi equation, see \cite{Ziliotto}.  

\medskip

In view of the remark at the beginning of the introduction and the simplifications above we study the long time average of the solutions  to  the initial value problem 
\begin{equation}\label{eqn: t HJeqn}
u_t + h(Du) - V(x) = 0 \ \hbox{ in } \ \real^d \times(0,\infty) \ \hbox{ with } \ u(\cdot,0) \equiv 0. 
\end{equation}
We remind that the choice of initial data is due to our normalization $\xi_0 = 0$.

\begin{prop}\label{prop: non-hom} Assume \eqref{saddle}  and \eqref{takis-1}  and let $u$ be the solution of \eqref{eqn: t HJeqn} with random potential field $V$ satisfying \eqref{takis1} and Properties~\ref{prop: horizontal segment} and \ref{prop: vertical segment}.  Then, almost surely,
\begin{equation}\label{eqn: diff lims}
 \liminf_{t \to \infty} \frac{u(0,t)}{t} \leq -1+c_\delta \ \hbox{ and } \ \limsup_{t \to \infty} \frac{u(0,t)}{t} \geq 1-c_\delta.
 \end{equation}
\end{prop}

\subsection{The proof of Proposition~\ref{prop: non-hom} and the construction of the barriers}\label{sec: non-hom 2}
Since the arguments are symmetric, here we only prove the first part of \eqref{eqn: diff lims}.  The idea of the proof is that the potential $V$ creates ``obstacles'' for the propagation along vertical and horizontal directions, thus ``obstructing'' the averaging. The obstacles are quantified by special super-and sub-solutions, that is barriers, to \eqref{eqn: t HJeqn}. This comes to play by the existence of special sub-and super-solutions that force the claimed inequality between the largest and smallest possible long time averaged limits.

\medskip

{\it The proof of Proposition~\ref{eqn: diff lims}}: Fix $\theta >0$, $\omega\in \Omega$ and $n \geq 1$ so that Property~\ref{prop: horizontal segment} holds.  Then $V=-1$ on a horizontal segment of length $\frac{4}{\delta}T_n$ centered at $X=(X_1,X_2)$ with $|X| \leq \theta T_n$.

\medskip
 
 It turns out, as it is shown below, that  
 $$ v_+(x,t) := - (1-c_\delta)t +|x_2-X_2| +(\delta|x_1-X_1| + 2(t-T_n))_+$$
 is a super-solution to 
\begin{equation}\label{eqn: suprsoln}
v_t + H(Dv,x) \geq 0 \ \text{in } \ \R^2\times (0, T_n) \quad \hbox{ and } \quad  v(x,0)\geq 0.
\end{equation}
It follows that
\begin{equation}\label{takis3}
\frac{u(0,T_n)}{T_n} \leq \frac{v_+(0,T_n)}{T_n} \leq \frac{-(1-c_\delta)T_n+(1+\delta)\theta T_n}{T_n}\leq -1+c_\delta + (1+\delta)\e.
\end{equation}
Since there are almost surely infinitely many $T_n$ as above 
$$ \liminf_{t \to \infty} \frac{u(0,t)}{t} \leq -1+c_\delta+(1+\delta)\theta \ \hbox{ almost surely.}$$
Letting $\theta \to 0$ yields 
$$ \liminf_{t \to \infty} \frac{u(0,t)}{t} \leq -1+c_\delta \ \hbox{ almost surely.}$$
\smallskip

Next  we  check that $v_+$ satisfies \eqref{eqn: suprsoln}.  That  $v(x,0)\geq 0$ for all $x\in \R^2$ is immediate.

\medskip

    \begin{figure}[t]\label{fig: barrier}
 \centering
 \def\svgwidth{4in}
  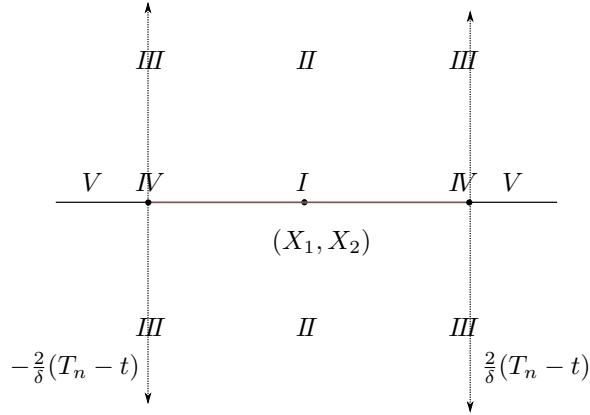
  \caption{The domain of the barrier function $v_+$ is split up into regions.}
 \end{figure}

\medskip

 For the super-solution property in $\R^2\times (0, T_n),$ we need to identify the sub-differential of $v_+$, 
\[ D^-v_+(x,t)= D^-_{x}v_+(x,t)\times \partial_t^-v_+(x,t).\]

Recall that, for a continuous $f:\R^m \to \R$,  the sub-differential  $D^-f(z)$ is
\[D^-f(z):= \{ p \in \R^m :    f(z') \geq f(z) + p\cdot (z'-z) +\text{o}(|z'-z|) \}.\]

We split $\R^2 \times (0, T_n)$ into five subregions depending on the location of $(x,t) \in \R^2\times (0, T_n)$; see Figure~\ref{fig: barrier} for reference.

\medskip

\emph{Region $I$:} The line segment $x_2 = X_2$, $|x_1-X_1| < \frac{2}{\delta}(T_n-t)$ and $t\in (0,T_n)$. Here $V(x) = -1$ and
\[ D^-_{x}v_+(x,t)=\{0\} \times [-1,1] \ \hbox{ and } \ \partial_tv_{+}(x,t)=-1+c_\delta.\]
 Then \eqref{eqn: h in square} yields the super-solution condition, 
\[  -1+c_\delta + h( \xi) - V(x) = h(\xi)+c_\delta  \geq 0 \ \hbox{ for } \ \xi \in \{0\} \times [-1,1]. \]

\medskip

\emph{Region $\II$:}  Off the line segment, $x_2 \neq X_2$ and $|x_1-X_1| < \frac{2}{\delta}(T_n-t)$ and  $t\in (0,T_n)$.   Here  
\[ D_{x}v_+(x,t) = \textup{sgn}(x_2-X_2) e_2 \ \hbox{ and } \ \partial_tv_{+} = -1+c_\delta.\]  
Since  $V \leq 1$ on $\real^d$, it follows that,
\[ \partial_tv_{+} + H(D_{x}v_+,x) = -1+c_\delta + h( \pm e_2) - V(x)  \geq -1+3 - V(x) \geq 3-2 = 1>0.\]

\emph{Region $\III$:}  The vertical lines $|x_1-X_1| = \frac{2}{\delta}(T_n-t)$ except for $x_2 = X_2$ and $t\in (0,T_n)$. We have
\[ D^-v_+(x,t) \subseteq [-\delta,\delta] \times \{1,-1\} \ \hbox{ and } \ \partial_t^-v_{+}(x,t)  \subseteq [-1,1+c_\delta] .\]  The choice of $\delta$ yields that $h(\xi) \geq 2$ for $\xi \in [-\delta,\delta] \times \{1,-1\}$. It follows that, for every $(\xi,s) \in D^-v_+(x,t)$,
\[ s + H(\xi,x) \geq -1 + h(\xi) - V(x) \geq h(\xi)-2 \geq 2-2 =0.\]

 \emph{Region $\IV$:}  The intersection of the horizontal segment $x_2 = X_2$ and the vertical lines $|x_1-X_1| = \frac{2}{\delta}(T_n-t)$ for all $t\in (0,T_n)$. We have $V(x) = -1$ and 
\[ D^-_x v_+(x,t) = [-\delta,\delta] \times [-1,1] \ \hbox{ and } \ \partial^-_t v_+(x,t) = [-1+c_\delta,1+c_\delta].\]
Now \eqref{eqn: h in square} gives $h(\xi) \geq -c_\delta$ for all $\xi \in D^-_xv_+(x,t)$, and,  for every $(\xi,s) \in D^-v_+(x,t),$ we get the supersolution condition
\[ s + H(\xi,x) \geq -1+c_\delta -c_\delta -V(x) = 0,\]
and, finally,
\smallskip

 \emph{Region $V$:} The region  $x\in \{(y_1,y_2) \in \R^2 :|y_1-X_1| > \frac{2}{\delta}(T_n-t)\}$ for all  $t\in (0,T_n)$ .   We know that $V(x) \leq 1$ and,
 \[D^-_x v_+(x,t) \subseteq [-\delta,\delta] \times [-1,1] \ \hbox{ and } \  \partial_t^- v_+(x,t) =\{ 1+c_\delta\}.\]
Since \eqref{eqn: h in square} gives  $h(\xi) \geq -c_\delta$ for all $\xi \in D^-_xv_+(x,t)$,  for every $(\xi,s) \in D^-v_+(x,t),$
\[ s + H(\xi ,x) \geq 1+c_\delta -c_\delta -V(x) \geq 1-V(x) \geq 0.\]

\subsection{The construction of  the random field}\label{sec: Constructing the random field}  We present here the construction of the random field $V$ in $\R^d$ following the ideas in \cite{Ziliotto}.   

\medskip

We fix an intermediate dimension $m \in \{1,\dots,d-1\}$ and differentiate between the $m$-dimensional ``horizontal" subspace spanned by $e_1,\dots,e_m$ and the $d-m$-dimensional ``vertical" subspace spanned by $e_{m+1},\dots,e_d$, and,   for $x \in \real^d,$ we set $x^1: = x_1e_1 + \cdots x_m e_m$ and $x^2: = x_{m+1}e_{m+1} + \cdots x_d e_d$.

\medskip

For each $k \in \integer^d,$ 
let $X_k, Y_k$ be i.i.d. dyadic, that is  $\mathbb{D}: = \{ 2^n : n \in \mathbb{N}\}$ valued,  random variables with distribution, for each $n$, $ \P(X_k = 2^n) = \alpha_d2^{-dn}$ with $  \alpha_d = (2^d-1).$

\medskip

For concreteness we take the  probability space to be $\Omega: = (\mathbb{D}\times \mathbb{D})^{\integer^d}$
and denote by  $\mathcal{F}$ the $\sigma$-algebra generated by cylinder sets. 

\medskip

We give an informal description of the construction, which will then be made precise.  Consider placing at each $k \in \integer^d$ a ``horizontal" cube (parallel to $\real^m \times \{0\}^{d-m}$) centered at $k$ of side length $X_k$ and a ``vertical" cube (parallel to $\{0\}^m \times \real^{d-m}$) centered at $k$ of side length $Y_k$.  Now we remove some of these cubes so that there are no intersections between horizontal and vertical ones.  We go through each of the $X_k$ and $Y_k$ (the order is irrelevant), and, if the horizontal (resp. vertical) cube centered at $k$ intersects any vertical (resp. horizontal) cube with larger (or equal) side length, we ``mark it'' for  deletion.  After going through all the horizontal and vertical cubes and making the required markings, we then remove all the cubes which have been marked for deletion. The remaining configuration of cubes will have no intersections between horizontal and vertical cubes.

\medskip

Now we make this construction precise.  We assign to each $k \in \integer^d$  the mark $M_k \in \{-1,0,1\}$ according to the following rules:
\begin{align}
M_k&=-1 \ \hbox{ if $Y_\ell < \max\{X_k, 2|(\ell-k)^2|_\infty\}$,  for all $\ell \in \integer^d$ with $|(\ell-k)^1|_\infty \leq X_k/2,$} \label{eqn: mark -1}\\
M_k &= 1 \quad \hbox{ if   $X_\ell < \max\{ Y_k, 2|(\ell-k)^1|_\infty\}$, for all $\ell \in \integer^d$ with $|(\ell-k)^2|_\infty \leq Y_k/2$.} \label{eqn: mark 1}\\
M_k &= 0  \quad \hbox{ if neither of the two alternatives in \eqref{eqn: mark -1} and \eqref{eqn: mark 1} hold.} \label{eqn: mark 0}
\end{align}
In terms of the informal description, above $M_k = -1$ (resp. $M_k = 1$) signifies that the horizontal (resp. vertical) cube centered at $k$ is not deleted, while $M_k = 0$ signifies that both the horizontal and vertical cubes centered at $k$ are deleted.  Note that \eqref{eqn: mark -1} and \eqref{eqn: mark 1} are disjoint events, since they are subsets, respectively, of the disjoint events $ \{X_k > Y_k\}$ and  $ \{Y_k > X_k\}$. Moreover, the computations of the following section will verify that $\P(M_k = \pm 1)>0$, that is the construction results in something non-trivial.

\medskip

If  $M_k = -1$ (resp. $M_k = 1$), we  assign to $k\in \mathbb{Z}^d$ a horizontal (resp. vertical) cube $H_k$ (resp. $R_k$) of side length $X_k$ (resp. $Y_k$).  We denote by $H$ (resp. $R$) the union of all such $H_k$ (resp. $R_k$).  We set 
\[ V_-(x): = \min\{ -1+ 2d(x,H),0\} \quad \text{and} \quad V_+(x): = \max\{ 1- 2d(x,R),0\}.\]

\medskip

It follows from the definition of  $M_k$ that  $d(R,H) \geq 1$, and, hence, the sets  $\{V_+ > 0\}$ and $\{V_-<0\}$ are disjoint. 

\medskip

The potential $V$ is is then defined
$$V(x): = V_+(x) + V_-(x).$$
The random field $V$ constructed above is stationary with respect to $\integer^d-$translations. Moreover, since  $\P$ is a product measure, the $\integer^d$ translation action on $\Omega$ is ergodic.  Later (subsection~\ref{sec: space-time}) we will need to construct 
a random field which is stationary with respect to $\real^d$ translations and ergodic.  This is a standard construction; see, for example, the book of Jikov, Kozlov and Oleinik \cite{JikovKozlovOleinik}. The translation  $\tau$ is taken to be a uniform random point of $[0,1)^d$, which is  independent of  $X_k$ and $Y_k$,  and 
$ \overline {V}(x) = V(x-\tau).$
This new random field is $\real^d-$stationary and remains ergodic.  We note that this construction does not  preserve strong mixing, a limitation which is  relevant later in the paper.

 \subsection{The properties of the random field}\label{sec: Properties of the random field}
We show that the random field constructed in the previous subsection satisfies Property~\ref{prop: horizontal segment} and Property~\ref{prop: vertical segment}.  Since they  are symmetric, we will only check Property~\ref{prop: horizontal segment}.  Although we follow some of the arguments in \cite{Ziliotto}, it is,  however,  necessary to use extra care when  applying to converse of the Borel-Cantelli Lemma since the events involved are not independent.

 \medskip
 
We take $T_n:=2^n$ and let  $ 0 < \theta \leq \frac{1}{2}$ be a dyadic rational. The  events $B_n$ related to Property~\ref{prop: horizontal segment} are 
 $$ B_n: = \{ \exists \ k\in \integer^d \hbox{ with }  |k| \leq \theta T_n \hbox{ and } V=-1 \hbox{ on an $m$-cube centered at $k$ of side length } \geq T_n\}.$$
The goal is  to show that $B_n$ occurs infinitely often.  Since they are not independent, it is not enough  to show that $\liminf \P(B_n) >0$.  Instead we need to use a converse version of the Borel-Cantelli Lemma for non-independent sets; see, for example,  Bruss \cite{converseBC} and Tasche\cite{BC2mixing}.  This requires to define the notion of $\alpha$-mixing. 
 \begin{DEF}
 Let $(Z_n)_{n \in \mathbb{N}}$ be a sequence of random variables on a common probability space. The $\alpha$-mixing coefficients are defined by 
 $$\alpha(n): = \sup \{ |\P(E \cap F) - \P(E)\P(F)| : k \in \mathbb{N}, \ E \in \sigma(Z_1,\dots,Z_k) \ \hbox{ and } \ F \in \sigma(Z_{k+n})\}.$$
 If $\alpha(n) \to 0$ as $n \to \infty$,  the sequence $Z_n$ is called $\alpha$-mixing.
 \end{DEF}
 The result we use is the following.
 \begin{thm}[\cite{BC2mixing}]\label{thm: bc2 mixing}
 Let $(A_k)_{k \in \mathbb{N}}$ be a sequence of events with non-increasing probabilities $\P(A_k)$. 
 Then $\P(A_k \hbox{ infinitely often})= 1$ is implied by either of the following two conditions:
 \begin{enumerate}[(i)]
 \item There exists $r \in  [-1,\infty)$ such that
 $$ \sum_{n=1}^\infty n^{-\frac{1}{r+2}}\P(A_n) = +\infty \ \hbox{ and } \ \sum_{n=1}^\infty n^r \alpha(n) < \infty.$$
 \item \label{part: ii} There exists $b \in (0,1)$ such that,
 $$ \sum_{n=1}^\infty \frac{\P(A_n)}{1+\log n} = +\infty \ \hbox{ and } \ \alpha(n) = O(b^n).$$
 \end{enumerate}
 \end{thm}
 
Next we formulate, in terms of the random variables $X_k^j$, a sufficient condition for the event $B_n$ to occur. For this, it is convenient  to assume  that $n \geq n_\theta$ sufficiently large so that $\theta T_{n-1} \geq 1$.  The claim is that, in order for $B_n$ to occur, it suffices that 
\begin{equation*}
\begin{cases}
\text{$(1)$ one of the $X_k$'s with $|k| \leq \theta T_n$ has value $T_n$}\\[1mm]
\text{and}\\[1mm]
\hbox{$(2)$ for every $ |k^1|_\infty \leq T_n$ \ $Y_k < \max\{T_n, |k^2|_\infty\}$.}
\end{cases}
\end{equation*}
Note that if $(1)$ holds, there exists $|k_0| \leq \theta T_n$ with $X_k \geq T_n$. For $B_n$ to occur,  we only  need to check that $M_{k_0}=-1$. Indeed, since $\theta \leq 1/2$, and, hence, $\theta/(1-\theta) \leq 1$, the fact that $(2)$ holds for every $ |k^1|_\infty \leq T_n$, and, thus, for every $|(k-k_0)^1|_\infty \leq \frac{1}{2}T_n$,
yields
 \begin{align*} 
 Y_k &<\max\{T_n, |k^2|_\infty\}  \leq \max\{T_n, |(k-k_0)^2|_\infty+|k_0^2|_\infty\}\leq \max\{T_n, |(k-k_0)^2|_\infty+\theta T_n\} \\
 & \leq  \max\{T_n, (1+ \tfrac{\e}{1-\e}) |(k-k_0)^2|_\infty\} \leq \max\{T_n,2|(k-k_0)^2|_\infty\}.
 \end{align*}
 Dealing with  condition $(2)$ is more delicate and we discuss this later.

\medskip

 Next  we consider the events $C_n$ and  $D_n$ which formalize conditions $(1)$ and $(2)$ respectively and together guarantee the occurence of $B_n$. For this, it is convenient to incorporate some independence in the $C_n$.
\medskip
 
 We have:
 \begin{equation}
  C_n := \{ \exists \ k\in \integer^d \hbox{ with } \theta T_{n-1} <  |k|_\infty \leq \theta T_n \hbox{ and } X_k = T_n\},
  \end{equation}
 and
  \begin{equation}
  D_n: = \{ \forall \ k\in \integer^d \hbox{ with } |k^1|_\infty \leq T_n \hbox{ it holds } Y_k < \max\{T_n, |k^2|_\infty\}\}. 
  \end{equation}
Note that the restriction $\theta T_{n-1}<|k|_\infty \leq \theta T_n$ in $C_n$ makes the $C_n$ independent in $n$.  

\medskip  

The following lower bounds on the probability of $C_n$ and $D_n$ are the same as in \cite{Ziliotto}. The proof is presented later in the subsection.

  \begin{lem}\label{lem: prob lower bound}
Let $C_n$ and $D_n$ be defined as above.  Then
$$\inf_{n\geq n_\theta} \P(C_n), \inf_{n\geq n_\theta} \P(D_n)  >0 \ \hbox{ and  } \ \inf_{n \geq n_\theta} \P(C_n \cap D_n) >0 .$$
\end{lem}
 
Just the lower bounds are not sufficient to apply the converse Borel-Cantelli and conclude that, almost surely, $C_n \cap D_n$ occurs infinitely often.  The events $C_n$ are mutually independent, since they are defined in terms of disjoint collections of the $X_k$. The same is true for $\{\{C_n\}_{n=1}^\infty, D_j\}$, since the $Y_k$ and $X_k$ are independent of each other. This is not, however,  the case for the $D_n$'s and to conclude we need to establish the $\alpha$-mixing property discussed earlier in the subsection.
\begin{lem}\label{lem: mixing prop}
Let $n,\ell \in \mathbb{N}$.  For any $E \in \sigma(D_1,\dots, D_\ell)$ and $F \in \sigma(D_{\ell+n})$, 
$$ |\P(E\cap F) - \P(E)\P(F)| \leq C_{d,m}2^{-nm}.$$
\end{lem}
The mutual independence of the $(C_n)_{n=1}^\infty$ and their independence from the $D_n$ yield  the same $\alpha$-mixing property for the sequence $C_n \cap D_n$. 

\medskip

It is immediate from Lemma~\ref{lem: prob lower bound}  that
\[ \sum_{n=1}^\infty \frac{\P(C_n \cap D_n)}{1+\log n} =+\infty.\]
We remark that there is some leeway here to allow for some decay of $\P(C_n \cap D_n)$ while maintaining that the sum is infinite.  Since, in view of  Lemma~\ref{lem: mixing prop}, the events $C_n \cap D_n$ are $\alpha$-mixing with exponential tails, we can apply Theorem~\ref{thm: bc2 mixing}~(\ref{part: ii}) to conclude that
\begin{equation}
 \P(B_n \ \hbox{infinitely often}) \geq \P(C_n \cap D_n \ \hbox{infinitely often}) = 1.
 \end{equation}
  
  \medskip
  
  \textit{Proof of Lemma~\ref{lem: mixing prop}.} We claim that $E \in \sigma(D_{\ell+n})$ and $F \in \sigma(D_1,\dots D_\ell)$ are conditionally independent on $\sigma(I_{\ell,n})$, where $I_{\ell,n}$ is the event
  $$ I_{\ell,n}: = \{ \forall \ k \in \integer^d \hbox{ with } |k^1|_\infty \leq T_\ell \hbox{ and } Y_k < \max\{T_{\ell+n}, |k^2|_\infty\} \}$$
  which has probability close to $1$.  
  
  \medskip
  
First we note that the facts that $D_j \in \sigma((Y_k)_{|k^1| \leq T_j})$ and the  $Y_k$'s  are i.i.d. imply that $D_1,\dots,D_\ell$ are independent of $(Y_k)_{|k^1|_\infty > T_\ell}$.

 \medskip
 
To  check the conditional independence, we first compute $\P(E,F| I_{\ell,n})$ with $E = D_n$ and for this we define 
$$D_{\ell,n}: =  \{ Y_k < \max\{T_{\ell+n}, |k^2|_\infty\} \  \text{holds for all}  \ k\in \integer^d \hbox{ with } T_\ell < |k^1|_\infty \leq T_{\ell+n}\}, $$
and  note that $D_n \cap I_{\ell,n} = D_{\ell,n} \cap I_{\ell,n}$. 
 
\medskip

Since $D_{\ell,n} \in \sigma((Y_k)_{|k^1|_\infty > T_\ell})$,   it is independent of $F$,  and we find  
$$ \P(D_{\ell+n}, F | I_{\ell,n} ) =  \P( D_{\ell,n} , F | I_{\ell,n}) = \P(D_{\ell,n}|I_{\ell,n})\P(F|I_{\ell,n})= \P(D_{\ell+n}|I_{\ell,n})\P(F|I_{\ell,n}). $$
Conditioning on $I_{\ell,n}^C$ we note that $I_{\ell,n}^C \subseteq D_{\ell+n}^C$ and, hence,
$$ \P(D_{\ell+n}, F | I_{\ell,n}^C ) = 0 = \P(D_{\ell+n} | I_{\ell,n}^C)\P(F|I_{\ell,n}^C). $$
Similar arguments show that $D_{\ell+n}^C$ is independent of $F$ conditional on $\sigma(I_{\ell,n})$.

 \medskip

Next we establish that there exists a  constant $C_{d,m}$, which depends only on $m,$ and $d$ such that 
$$\P(I_{\ell,n}^C) \leq C_{d,m} (T_\ell/T_{\ell+n})^m.$$
This follows from the following series of calculations that estimate $\P(I_{\ell,n})$  using  that $ (1-x) \geq e^{-2x}$ for $0< x\leq 1/2$ and the fact that  
there is a dimensional constant $C_{d,m}$ such that, for all $R\geq 1$, 
$$\sum_{\ell \in \mathbb{Z}^{d-m}, |\ell|_\infty \geq R}\frac{1}{|\ell|_\infty^d} \leq C_{d,m}R^{-m}.$$ 
We have:
 \begin{align*}
   \P(I_{\ell,n}) = &\prod_{|k^1|_\infty \leq T_\ell} \left(1-\frac{1}{\max\{|k^2|_\infty,T_{\ell+n}\}^d}\right) 
   \geq \exp\left(-2\sum_{|k^1|_\infty \leq T_\ell} \frac{1}{\max\{|k^2|_\infty,T_{\ell+n}\}^d}\right)\\ 
    &= \exp\left(-2T_\ell^m \sum_{z \in \mathbb{Z}^{d-m}}\frac{1}{\max\{|z|_\infty,T_{\ell+n}\}^d}\right)\\ 
& = \exp\left(-2T_\ell^mT_{\ell+n}^{-d}T_{\ell+n}^{d-m}- 2T_\ell^m\sum_{\ell \in \mathbb{Z}^{d-m}, |\ell|_\infty \geq T_{\ell+n}}\frac{1}{|\ell|_\infty^d}\right) \\[1mm]
   & \geq \exp(-2(1+C_{d,m})(T_\ell/T_{\ell+n})^m) 
   \geq 1-C_{d,m}(T_\ell/T_{\ell+n})^m.
   \end{align*}

  \medskip
  
The conditional independence  for $E \in \sigma(D_1,\dots, D_\ell)$ and $F \in \sigma(D_{\ell+n})$ yields 
  \begin{align*}
   \P(E \cap F) &= \E[\E( \indicator_E \indicator_F | \sigma(I_{\ell,n}))] =  \E[\E( \indicator_E| \sigma(I_{\ell,n}))\E( \indicator_F | \sigma(I_{\ell,n}))] \\[1mm]
   &=\E[ \P(E  | I_{\ell,n})\P(F  | I_{\ell,n})\indicator_{I_{\ell,n}}+ \P(E | I_{\ell,n}^C)\P(F | I_{\ell,n}^C)\indicator_{I_{\ell,n}^C}] \\[1mm]
   &= \P(E  | I_{\ell,n})\P(F  | I_{\ell,n}) \P(I_{\ell,n}) +  \P(E | I_{\ell,n}^C)\P(F | I_{\ell,n}^C)\P(I_{\ell,n}^C) 
   \end{align*}
   Using the previous decomposition and the bound on the probability of $I_{\ell,n}^C$, we find 
   $$ | \P(E \cap F) -  \P(E)\P( F)| \leq 2\P(I_{\ell,n}^C) \leq C_{d,m} (T_\ell/T_{\ell+n})^m = C_{d,m}2^{-md}.$$
\vskip-.1in   \qed

  \medskip
We continue with: 
 \medskip
 
  \begin{proof}[Proof of Lemma~\ref{lem: prob lower bound}.]  We check now, as in  \cite{Ziliotto}, that $\inf_{n\geq 1} \P(C_n)>0$ and  $\inf_{n\geq 1} \P(D_n)  > 0$.

  \medskip
  
The lower bound  for $\P(C_n)$ follows from 
  \begin{equation}
   1-\P(C_n) = (1-\P(X_0 = T_n))^{\theta^d(T_n^d-T_{n-1}^d)}= (1-\frac{\alpha_d}{T_n^d})^{\theta^dT_n^d(1-2^{-d})} \leq \exp(-\alpha_d^2\theta^d).
   \end{equation}
     Note that $n \geq n_\theta$, so that $\theta T_{n-1} \geq 1$, guarantees that $\theta T_n$ and $\theta T_{n-1}$ are both integers.
     
     \medskip
     
   For the lower bound for $\P(D_n),$ we use that there exists a dimensional constant $C_{d,m}$ such that, for all  
   $R\geq 1$, 
   $\sum_{\ell \in \mathbb{Z}^{d-m}, |\ell|_\infty \geq R}\frac{1}{|\ell|_\infty^d} \leq C_{d,m}R^{-m},$
 and we proceed with the following estimates:

  \begin{align*}
   \P(D_n) &= \prod_{|k^1|_\infty \leq T_n} \left(1-\frac{1}{\max\{|k^2|_\infty,T_n\}^d}\right) 
   \geq  \exp\left(-2\sum_{|k^1|_\infty \leq T_n} \frac{1}{\max\{|k^2|_\infty,T_n\}^d}\right) \\
   & \geq \exp\left(-2T_n^m \sum_{\ell \in \mathbb{Z}^{d-m}}\frac{1}{\max\{|\ell|_\infty,T_n\}^d}\right) \\
   & =  \exp\left(-2T_n^mT_n^{-d}T_n^{d-m}- T_n^m\sum_{\ell \in \mathbb{Z}^{d-m}, |\ell|_\infty \geq T_n}\frac{1}{|\ell|_\infty^d}\right) 
    \geq \exp(-2(1+C_{d,m})).
   \end{align*}
   \end{proof}

\subsection{The mixing properties of the random field $V$}\label{sec: mixing prop}  We investigate the mixing properties of  $V$ and work only with its $\integer^d$-stationary version.   The reason is that the construction  of the $\real^d$ stationarity described earlier yields  a random field which is  ergodic but not mixing.  We remark that one could probably maintain all the desired properties of $V$ including both $\real^d$-stationarity and the mixing properties proven below in Lemma~\ref{lem: mix est} by basing the definition of $V$ on vertical and horizontal line/plane segments placed on an independent Poisson family of points rather than the $\integer^d$ lattice points.

\medskip

The total variation based mixing distance between two $\sigma$-algebras $\mathcal{G}_1, \mathcal{G}_2$ is defined
\begin{equation}
\alpha(\mathcal{G}_1,\mathcal{G}_2): = \sup \{  |\P(E \cap F) - \P(E)\P(F)|  : \ E \in \mathcal{G}_1 \ \hbox{ and } \ F \in \mathcal{G}_2 \}.
\end{equation}
For $A\subset \R^d$ a Borel set,  $\mathcal{F}_A$ is the $\sigma$-algebra generated by the random variables $(V(x))_{x \in A}$.  The $\alpha-$mixing coefficients associated with the random field $V$ are given by 
\begin{equation}
\alpha(r): = \sup \{\alpha(\mathcal{F}_A,\mathcal{F}_B) : \hbox{ with } A, B \subset \real^d \hbox{ and } \inf_{x \in A, y \in B} |x-y| \geq r\}.
\end{equation}

The following result was shown in \cite{Ziliotto} for $d=2$. 
\begin{thm}
The strong mixing coefficient $\alpha(r)$ associated with the random field $V$ defined in Section~\ref{sec: Constructing the random field} does not converge to $0$ as $r \to \infty$.  
\end{thm}
In this section we show that $V$ actually has a mixing property which is weaker than $\alpha$-mixing.  Although it does not seem to be a sufficiently strong  to be used to  prove homogenization, we make note of it here.  
 
\medskip

The following notion of mixing was introduced in Bramson, Zeitouni and Zerner \cite{BramsonZeitouniZerner}. Given $\gamma >0$, the random field $V$ is said to be polynomially mixing with order $\gamma$, if, for every finite subset $\Lambda$ of $\integer^d$ with fattening $\hat\Lambda = \cup_{k \in \Lambda} ([-1/2,1/2]^d+k)$,
\begin{equation}\label{eqn: bz mixing}
\sup_{\ell \in \integer^d\setminus \{0\}} |\ell|^\gamma\alpha(\mathcal{F}_{\hat\Lambda},\mathcal{F}_{\hat\Lambda+\ell}) <+\infty.
\end{equation}
In the next lemma we prove two mixing estimates for $V$. The second, which is a corollary of the first, implies that $V$ is polynomially mixing with order $\gamma = \frac{dm}{d+m}$. 
\medskip

After our work was completed, we were informed that a mixing estimate similar to part~(\ref{part: two}) below has also  been derived by Armstrong and Ziliotto \cite{ArmstrongPC}.  

\begin{lem}\label{lem: mix est}
(i)~There exists  $C=C(d,m)>0$  such that, for all finite $\Lambda \subset \integer^d$ and  $R>1$,
\begin{equation}\label{part: one} 
 \alpha(\mathcal{F}_{\hat\Lambda},\mathcal{F}_{(\hat\Lambda+Q_R)^C}) \leq C\sum_{\ell\in \Lambda} \frac{1}{(d(\ell,\hat\Lambda^C)+R)^{\frac{dm}{d+m}}} \leq \frac{C|\Lambda|}{R^{\frac{dm}{d+m}}}.
 \end{equation}
(ii)~There exists  $C(d,m)>0$ so that, for all finite $\Lambda \subset \integer^d$ and  $\ell \in \integer^d$,
\begin{equation}\label{part: two} 
 \alpha(\mathcal{F}_{\hat\Lambda},\mathcal{F}_{\hat\Lambda+\ell}) \leq C|\Lambda|\frac{1}{1+(|\ell|-\textnormal{diam}(\hat\Lambda))_+^{\frac{dm}{d+m}}}.
 \end{equation}
\end{lem}

It will be useful in the proof to define a version of $V$ with localized dependence on the coefficients $(X_k,Y_k)$.  Given a subset (finite or infinite) of the lattice $\Lambda \subseteq \integer^d,$ we define $V^\Lambda$ by the same process as in subsection~\ref{sec: Constructing the random field}, except using only the $(X_k,Y_k)$ with $k \in \Lambda$, that is giving mark $M_k = 0$ for every $k \notin \Lambda$.

\medskip

\begin{proof}[{Proof of Lemma~\ref{lem: mix est}.}]   
In order to keep things simple, it  is convenient to view $V$ as defined just on the integer lattice $\integer^d$. 
\medskip

Let 
 $V_* : \integer^d \to \{-1, 0, 1\}$ to be the restriction of $V$ to the integer lattice, and, for any set $\Lambda \subset \integer^d$, set 
$$\F_*(\Lambda): = \sigma(V_*(\ell): \ell \in \Lambda).$$
The mixing properties for $V_*$ yield the ones of $V$, since,  given  $x \in \real^d$, there exists $k \in \integer^d$  such that $x \in Q_x := k+[0,1)^d$, and  then
\begin{equation}
V(x) \in \sigma(V_*(z):  z \ \hbox{ a corner of the lattice cube containing $Q_x$}).
\end{equation}
Indeed observe that out of the two ``horizontal" and two ``vertical" faces of $Q_x$, there is at most one horizontal face $H_x$ such that $d(x,H_x) < 1/2$ and  at most one vertical face $R_x$ such  such  that $d(x,R_x) < 1/2$.  Then 
$$ V(x) = 1-2d(x,R_x) \ \hbox{ if and only if, for all the corners $z$ of $R_x$,  $V_*(z) = 1$,}$$
and
$$ V(x) = -1+2d(x,H_x) \ \hbox{ if and only if, for all the corners $z$ of $H_x$,  $V_*(z) = -1$,}$$
\vskip.075in

and, if neither of the above events hold, then $V(x) = 0$.  It follows  that $V(x)$ is a function of $V_*(z)$ for $z$ which are corners of the lattice cell $Q_x$.

\medskip

We outline next the proof of the mixing properties for $V_*$.  For each $\ell \in \integer^d$ and $L\geq 5$, we construct an event $I(\ell,L)$ such that $V_*(\ell)$ conditioned upon $I(\ell,L)$  is independent of all the $X_k, Y_k$ for $|k-\ell| \geq  L$.  Also for each $\ell \in \integer^d$, $L \geq 5$ and $R \geq L+5$, we construct an event $J(\ell, L,R)$ such that,  for every $|x-\ell| \geq R$, $V(x)$ conditioned upon $J(\ell, L,R)$ is independent of all the $X_k,Y_k$ with $|k| < L$. Then, we choose  $\lambda(R)$  to maximize the probability of the event 
$$K(\ell,R): = I(\ell,(1+\lambda)^{-1}R) \cap J(\ell,(1+\lambda)^{-1}R,R) \ \hbox{ for } \ \ell \in \integer^d \ R\geq 5.$$
Conditioning on $K$, an event which will have probability close to $1$, we obtain the desired mixing estimates.

\medskip

To simplify  the notation, we assume that the events $I(\ell,L)$ and $J(\ell,L,R)$ are centered at $\ell = 0$, since the extension to $\ell\neq 0$ follows by  the stationary translations.  Next  we  write   $I,J$ to mean $I(0,L)$ and $J(0,L,R)$ respectively. 
\medskip

Let $L \geq 5$ and $R \geq L+5$ and define $J(0,L,R)$ by 
\begin{equation}
J(0,L,R) := \{ \forall \ |\ell| < L, \   X_{\ell} \vee Y_\ell < (R-L)\}.
 \end{equation}
Computations similar to the ones in the previous subsection, which we do not repeat here, yield the lower bound
\begin{equation}\label{eqn: J prob est}
 \P(J) \geq (1-c_d(R-L)^{-d})^{c_dL^d} \geq 1-\frac{c_d}{(\frac{R}{L}-1)^d}.
 \end{equation}

\medskip

To make precise the meaning of the conditional independence, we define the conditional probability measure $\P(\cdot |J)$ on $\mathcal{F}$ by 
\vskip.025in
$$\P(E | J) := \P(E \cap J) / \P(J) \ \hbox{ for } \ E \in \mathcal{F}.$$
\vskip.075in
We emphasize this most elementary definition of conditional probability in order to make clear the fact  that we are not using the typical measure theoretic conditioning.  
\medskip

The claim is that, for all $E \in \mathcal{F}_*(\integer^d \setminus Q_R)$ and $F \in \sigma(X_k,Y_k: \ k \in Q_L)$,
\begin{equation}\label{eqn: cond ind J}
\P(E \cap F| J) = \P(E|J)\P(F|J).
\end{equation}
We also make use of the localized version of the coefficient field $V_*^\Lambda$ defined for any set $\Lambda \subset \integer^d$, which,  by its construction, is  $\sigma(X_k,Y_k : k \in \Lambda)$ measurable.  
\medskip

The claim is that
\begin{equation}\label{eqn: equal to localized J}
 V_*^{\integer^d \setminus Q_L}(\ell)\indicator_{J} = V_*(\ell)\indicator_{J} \ \hbox{ for all } \ell \in \integer^d \setminus Q_R.
 \end{equation} 
 \vskip.025in
The conditional independence \eqref{eqn: cond ind J} then follows immediately from \eqref{eqn: equal to localized J}.
 
 \medskip
 
If $|\ell| \geq R$, then either $|\ell^1| \geq R$ or $|\ell^2| \geq R$.  Since the arguments are similar,  here we assume that $|\ell^1| \geq R$. Let $H_k:=k+[-X_{k/2}, X_{k/2}]^m \times \{0\}^{d-m}$ be a horizontal segment which is centered at some $k \in \integer^d$ and passes through $\ell$. In order for $H_k$ to be marked by a segment from $Q_L$,  it must intersect the set $\{z : |z^1| \leq L\}$ and so it is necessary that $X_k \geq (R-L)$. On the event $J$, for any $k \not\in Q_R$ the segment $H_k$ cannot be marked for deletion by any of the $Y_j$ for $j \in Q_L$ since $Y_j < X_k$.  If $V_k$ is a vertical segment passing through $\ell$, then $k^1 = \ell^1$ and so $|k^1| \geq R$. For a horizontal segment centered at a point of $Q_L$ to delete $V_k$ it needs to have length at least $2(R-L)$ which is again not possible on the event $J$.  It follows  that $V_*(\ell) \indicator_J=V_*^{\integer^d \setminus Q_L}(\ell)\indicator_{J}$.

\medskip

To define $I(0,{L})$ we need  the following  four events:
\begin{equation*}
\begin{array}{l}
I_h :=
 \left\{ 
\begin{array}{ll}
\hbox{ for all $|k| \geq {L}$ such that $ k^2 \neq 0$} &\hbox{either } \ Y_{(0,\ell)} < 2\max\{|k^2-\ell|,|\ell|\} \ \hbox{for all $\ell \in \mathbb{Z}^{d-m}$,} \\[1mm] 
 &\hbox{or } X_k < \max\{2|k^1|,|k^2|\}
\end{array}\right\} \\[5.5mm]
\II_h: = \{ X_k < 2|k^1| \text{for all } \ |k| \geq {L} \hbox{ such that  } k^2 = 0 \} \\[1.5mm]
I_v: = \left\{\begin{array}{ll}
\hbox{ for all $|k| \geq {L}$ such that  $k^1 \neq 0$} &\hbox{either $X_{(\ell,0)} < 2\max\{|k^1-\ell|,|\ell|\}$ for all $\ell \in \mathbb{Z}^{m}$,} \\[1mm] 
 &\hbox{or } Y_k < \max\{2|k^2|,|k^1|\}
 \end{array}\right\} \\[5.5mm]
\II_v: = \{ Y_k < 2|k^2| \ \text{ for all} \ |k| \geq {L} \hbox{ such that } k^1 = 0\}.
\end{array}
\end{equation*}
Note that $I_h$ and $I_v$ are independent since they are defined in terms of disjoint sets of $X_\cdot, Y_\cdot$, and, moreover, 
$ \II_v \subset I_h \ \hbox{ and } \ \II_h \subset I_v.$
\medskip

It follows that  $I: = I_h \cap \II_h \cap I_v \cap \II_v = I_h \cap I_v$ is the intersection of two independent events. We remark that, although they appear redundant,  the events $\II_h, \II_v$ will have a different meaning in the proof of conditional independence. 
\medskip

In the sequel we make use of the estimate 
\begin{equation}\label{eqn: I prob est}
 \P(I) \geq 1-C{L}^{-m};
\end{equation}
its proof is presented at the end of the ongoing one. 

\medskip

To prove conditional independence, as with the argument for $J$,  we use the localization $V_*^{Q_L}$ which is independent of $X_k,Y_k$ for $k \in \integer^d \setminus Q_L$ and aim to show that,
\begin{equation}
V_*(0) = V_*^{Q_L}(0) \ \hbox{ on the event $I$}.
\end{equation}
Then, for any $E \in \mathcal{F}_*(\{0\})$,  we can write $E\cap I = E' \cap I ,$  where $E'$ is an event in $\sigma(V_*^{Q_L}(0)),$ and, hence, $\P(E \cap F|I)=\P(E' \cap F|I)$ and  $\P(E|I)= \P(E'|I)$. 
\medskip

Moreover, it follows from \eqref{eqn: I prob est}  that,  for any measurable set $A \in \Omega$, 
  $$|\P(A | I) - \P(A)|  = \P(A \cap I^C)/\P(I) \leq CL^{-m}.$$
Then, for any $F \in \sigma(X_k,Y_k: \ k \in \integer^d \setminus Q_L)$,
\begin{align*}
 |\P(E \cap F) - \P(E) \P(F)| &\leq |\P(E \cap F|I) - \P(E|I) \P(F)| + |\P(E \cap F) -\P(E \cap F|I)| + |\P(E) -  \P(E|I)| \P(F) \\
 &\leq  |\P(E' \cap F|I) - \P(E'|I) \P(F)| + C{L}^{-m} \leq  C{L}^{-m}. 
 \end{align*}

We show next that, on the event $I_h \cap \II_h$, neither can the origin be on any horizontal segment centered at some $k \in \integer^d \setminus Q_L$ nor can a horizontal segment centered a point $k \in \integer^d \setminus Q_L$ mark any vertical segment containing $0$.  
\medskip

It follows that,  on the event $I_h\cap \II_h$,   $V_*(0)$ is independent of the $(X_k)_{k \in \integer^d \setminus Q_L}$.  A symmetrical argument gives that $V_*(0)$ is independent of the $(Y_k)_{k \in \integer^d \setminus Q_L}$ on the event $I_v\cap \II_v$. 
\medskip

 The event $\II_h$ guarantees that the horizontal segments centered at $k \in (\integer^d \setminus Q_L) \cap \{ k^2 = 0\}$ do not contain $0$.  The other possibility is that a horizontal line segment centered at some $|k| \geq L$ with $k^2 \neq 0$ marks a vertical segment centered at a point $(0,\ell)$ for an $\ell \in \integer^{d-m}$ which crosses both $(0,k^2)$ and $0$.  This is not possible on the event  $I_h$.  Indeed on  $I_h$ either (i) there is no such vertical segment crossing both $(0,k^2)$ and $0$,  since this would require $Y_{(0,\ell)} \geq 2 \max\{|k^2 - \ell|,|\ell|\}$), or (ii) $X_k < \max\{2 |k^1|,|k^2|\}$.  If (ii) holds, then either $X_k < 2|k^1|$ or $X_k < |k^2|$.  In the former  case, the horizontal segment centered at $k$ does not cross the vertical axis and, hence, cannot mark any point $(0,\ell)$. In the latter case, any vertical segment passing through both $(0,k^2)$ and $0$ must have length at least $|k^2|$ and so it is longer than $X_k$ and will not be marked by $X_k$.  

\medskip

Since \eqref{eqn: I prob est} and \eqref{eqn: J prob est} give 
 $$ \P(K) \geq 1-C((1+\lambda)^mR^{-m}-\lambda^{-d}),$$
we choose  $\lambda(R) := R^{\frac{m}{d+m}}$  to maximize the right hand side, and, hence, we find
 $$  \P(K) \geq 1- CR^{-\frac{dm}{d+m}}.$$

 
Let $\Lambda$ be finite a subset of $\integer^d$.  Next we prove \eqref{part: one} and observe that  \eqref{part: two} follows from \eqref{part: one} since, if $|\ell| \geq \diam(\Lambda)$, then 
$$\Lambda + \ell \subset (\Lambda + Q_R)^C \ \hbox{ with } \ R = |\ell| - \diam(\hat\Lambda).$$

Let $E \in \mathcal{F}(V_*(z): z \in \Lambda)$, $F \in \mathcal{F}(V_*(z): z\in (\Lambda+Q_R)^C)$ and define $\hat K := \cap_{\ell \in \Lambda} K(\ell,d(\ell,\Lambda^c)+R)$. It follows  from \eqref{eqn: I prob est} and \eqref{eqn: J prob est} that 
$$\P(\hat K) \geq 1- C\sum_{\ell\in \Lambda} \frac{1}{(d(\ell,\Lambda^C)+R)^{\frac{dm}{d+m}}}.$$
Then, by a calculation similar to the one in Lemma~\ref{lem: mixing prop} and using that $E$ and $F$ are independent conditional on the event $\hat K$, we find, for some  $C>0$,
$$ |\P(E \cap F) - \P(E)\P(F)| \leq C\sum_{\ell\in \Lambda} \frac{1}{(d(\ell,\Lambda^C)+R)^{\frac{dm}{d+m}}}.$$

\medskip

We conclude with the proof of \eqref{eqn: I prob est}.  The key step is a lower bound on the probability of $I_h$--note that  a very similar computation works for $I_v$:
\begin{align*}
\P(I_h) &= \prod_{|k| \geq {L}, k^2 \neq 0} \P( \hbox{either }Y_{(0,\ell)} < 2\max\{|k^2-\ell|,|\ell|\}  \hbox{ for all $\ell \in \mathbb{Z}^{d-m}$ or } X_k < \max\{2|k^1|,|k^2|\}) \\
&= \prod_{|k| \geq {L}, k^2 \neq 0} (1-\P(\hbox{$ \exists \ell \in \mathbb{Z}^{d-m}$ such that  $Y_{(0,\ell)} \geq 2\max\{|k^2-\ell|,|\ell|\} $ and  $X_k \geq \max\{2|k^1|,|k^2|$\})}) \\
& = \prod_{|k| \geq {L}, k^2 \neq 0} (1- \P(X_k \geq \max\{2|k^1|,|k^2|\})(1-\P( Y_{(0,\ell)} < 2\max\{|k^2-\ell|,|\ell|\text{for all $\ell \in \mathbb{Z}^{d-m}$ }\}))) \\
& = \prod_{|k| \geq {L}, k^2 \neq 0} (1- \P(X_k \geq \max\{2|k^1|,|k^2|\})(1- \prod_{\ell \in \mathbb{Z}^{d-m}} \P(Y_{(0,\ell)} < 2\max\{|k^2-\ell|,|\ell|\}))). \\
\end{align*}
Before continuing with the full computation, we estimate
\begin{align*}
\prod_{\ell \in \mathbb{Z}^{d-m}} \P(Y_{(0,\ell)} < 2\max\{|k^2-\ell|,|\ell|\}) &\geq \prod_{\ell \in \mathbb{Z}^{d-m}}(1-\frac{C}{(|\ell| + |k^2 - \ell|)^d}) 
 \geq \exp( - \sum_{\ell \in \mathbb{Z}^{d-m}}\frac{C}{(|\ell| + |k^2 - \ell|)^d}) \\
& \geq \exp (- \sum_{\ell \in \mathbb{Z}^{d-m}}\frac{C}{(|\ell| + |k^2|)^d}) 
 \geq 1- C|k^2|^{-m},
\end{align*}
where for the third inequality we  used that, since $\frac{1}{2}(|k^2|-|\ell|)\leq  |k^2 - \ell|$,  we have $|\ell| + |k^2 - \ell| \geq \frac{1}{2}(|k^2|+|\ell|)$.  
\smallskip

Using this information in  the estimate for $\P(I_h)$ we find
\begin{align*}
\P(I_h) &\geq \prod_{|k| \geq {L}, k^2 \neq 0}\left(1- \P(X_k \geq \max\{2|k^1|,|k^2|\})\frac{C}{|k^2|^m}\right) 
 \geq \prod_{|k| \geq {L}} \left(1- \frac{C}{(1+|k^2|)^m(|k^1|+|k^2|)^d}\right) \\
& \geq \exp(- \sum_{|k| \geq {L}, k \in \mathbb{Z}^d} \frac{C}{(1+|k^2|)^m(|k^1|+|k^2|)^d}) 
 \geq 1- C{L}^{-m};
\end{align*}
in  the last line we used that 
\begin{align*}
 \int_{\real^d \setminus B(0,{L})} \frac{1}{(1+|x^2|)^m(|x^1|+|x^2|)^d} \ dx &\lesssim   \int_{\real^d } \frac{1}{(1+|x^2|)^{m}({L}+|x^1|+|x^2|)^{d}} \ dx \\
 & = \int_{\real^{d-m}} \frac{1}{(1+|x^2|)^{m}}\int_{\real^m} \frac{1}{({L}+|x^2|+|x^1|)^{d}} \ dx^1dx^2 \\
 & \lesssim \int_{\real^{d-m}}\frac{1}{(1+|x^2|)^{m}}\frac{1}{({L}+|x^2|)^{d-m}} \ dx^2 
  \lesssim {L}^{-m}.
 \end{align*}
 A similar computation for $\P(I_v)$ yields
 $$  \P(I_v) \geq 1-C{L}^{-(d-m)}.$$
 Finally,  using that $m \leq d-m$ by assumption, we get 
 $$ \P(I) = \P(I_h \cap I_v) = \P(I_h)\P(I_v) \geq (1- C{L}^{-m})(1-C{L}^{-(d-m)}) \geq 1- C{L}^{- m}.$$
\end{proof}
 \subsection{Space-time random fields}\label{sec: space-time} We  discuss here a simple consequence of Proposition~\ref{prop: non-hom}.  If the Hamiltonian $h \in C^2( \real^d)$ has super-linear growth at infinity and there exists $\xi_0$ such  that $D^2h(\xi_0)$ is non-degenerate and has both positive and negative eigenvalues, then there exists a space-time stationary ergodic $V=V(x,t)$ for which  homogenization does not take place  for the Hamiltonian
$$H(\xi,x,t) = h(\xi) - V(x,t);$$
that is,  if   $u$ is the solution of
$$ u_t + H(Du,x,t) = 0 \ \hbox{ in } \ \real^d \times (0,\infty) \ \quad \ u(x,0) = \xi_0 \cdot x,$$
then
$$ \liminf_{t\to\infty} \frac{u(0,t)}{t} < \limsup \frac{u(0,t)}{t} \ \hbox{ almost surely.}$$
In what follows  we work for simplicity in $d=2$.  The  assumption on $h$ yields the existence of $\xi_0 \in \real^2$ such  that the eigenvalues of $D^2 h(\xi_0)$ have opposite sign.  Let $p_0 := Dh(\xi_0)$.  Although $h$ does not have a strict saddle point at $\xi_0$, since there is no reason that $p_0=0$, the Hamiltonian
$$ g(\xi) = h(\xi) - p_0 \cdot \xi \ \hbox{ has a strict saddle-point at $\xi_0$}$$
and is still coercive; note that this the reason we assumed that $h$ has superlinear growth.   It then follows from Lemma~\ref{prop: non-hom} that there is a probability space $(\Omega,\mathcal{F},\P)$, a group of measure preserving transformations $(\tau_{x})_{x\in\real^d}$ acting on $(\Omega,\mathcal{F},\P)$ which is ergodic, and a stationary random field $V_0 : \real^d \times \Omega  \to  \real$ (we hide the dependence on $\omega$ as usual) so that the solution $v$ of,
$$ v_t + g(Dv) - V_0(x) = 0 \ \hbox{ in } \ \real^2 \times (0,\infty) \ \quad \ v(x,0) =  \xi_0 \cdot x,$$
satisfies
$$ \liminf_{t \to \infty} \frac{v(0,t)}{t} < \limsup_{t \to \infty} \frac{v(0,t)}{t} \ \hbox{ almost surely in } \ \P.$$
Now observe that 
$ u(x,t) : = v(x - p_0 t, t),$
solves 
$$ u_t  + g(Du) + p_0\cdot Du -V_0(x-p_0t)=0 \ \text{in} \ \R^2\times (0,\infty) \quad \  u(x,0)=\xi_0\cdot x. $$
We need to justify that $V(x,t):=V_0(x-p_0t)$ is indeed space-time stationary and ergodic.  For this we extend the measure preserving system $(\Omega,\mathcal{F},\P,(\tau_{x})_{x\in\real^d})$ by defining
$ \sigma_{(x,t)}  := \tau_{x-p_0t}.$
It is easily checked that $\real^d \times \real$ acts on $(\Omega,\mathcal{F},\P)$ in a measure preserving way via the transformations $\sigma_{(x,t)}$.  Furthermore the group of transformations remains ergodic.
Indeed, if $E \in \mathcal{F}$ is invariant under all the translations $\sigma_{(x,t)}$, then, in particular, it is invariant under all $\sigma_{(x,0)} = \tau_x$ and thus, by the ergodicity of the original system, $\P(E) \in \{0,1\}$.  Finally we just need to check that $V(x,t,\omega)$ is stationary, a fact that follows from the identities 
$$V(x,t,\omega) = V_0(x-p_0 t,\omega) = V_0(0,\tau_{x-p_0t} \omega) = V(0,0,\sigma_{(x,t)}\omega).$$

\medskip

The properties of  $v$ imply 
\begin{equation}\label{eqn: ineq at p0}
 \liminf_{t \to \infty} \frac{u(p_0t,t)}{t} < \limsup_{t \to \infty} \frac{u(p_0t,t)}{t} \ \hbox{ almost surely in } \ \P. 
 \end{equation}
Although it is not so obvious, \eqref{eqn: ineq at p0} is equivalent to the analogous statement at $u(0,t)$.  This is a consequence of a well known fact , usually stated for the approximate corrector problem without time dependence; see \cite{ArmstrongSouganidis} Lemma 5.1.  

Define
$$ H_*(\xi_0) := \liminf_{t \to \infty} \frac{u(0,t)}{t} \ \hbox{ and } \ H^*(\xi_0) = \limsup_{t \to \infty} \frac{u(0,t)}{t} $$
The uniform Lipschitz continuity of $u$ yields that  $H_*$ and $H^*$ are translation invariant and therefore, by ergodicity, are almost surely constant.
\medskip

%
It is now  immediate consequence of \eqref{eqn: Rt} and \eqref{eqn: ineq at p0} that $H_*(\xi_0) < H^*(\xi_0)$.

\section{Quantitative Homogenization for Hamiltonians with Star-shaped Sub-level Sets}\label{sec: hamiltonian assumptions quant}
We identify a new class of nonconvex Hamiltonians, namely $H$'s with  ``quantitatively'' star-shaped sub-levels, such that the Hamilton-Jacobi equation \eqref{eqn: HJ1} homogenizes in random media with finite range dependence.

 \medskip
 
Our arguments extend to study the limit, as $\e \to 0$, of the solutions to the ``viscous'' Hamilton-Jacobi equation,
\begin{equation}\label{eqn: viscous HJ2}
u_t^\e - \e \Delta u + H(Du^\e,\tfrac{x}{\e})= 0 \  \hbox{ in } \  \real^d \times (0,\infty)  \qquad
u(x,0) = u_0(x),
\end{equation}
always in a medium with finite range dependence, with Hamiltonians satisfying  a crossover of quantitative  star-shapedness and positive homogeneity. The analysis applies even to more complicated degenerate quasilinear versions of \eqref{eqn: viscous HJ2}. We do not state, however, any specific results.

\medskip

The goal of this section is to point out that the strong star-shapedness assumption \eqref{eqn: intro star-shapedness} or, more precisely, a quantification of it (see  \eqref{eqn: assumption strongest} below), and not homogeneity is in fact the key assumption which allows the arguments of \cite{ArmstrongCardaliaguet} to be carried out for first-order problems.  

\medskip

As already discussed in the introduction, the analysis relies heavily on the arguments in \cite{ArmstrongCardaliaguet}.  Instead of rewriting large portions of \cite{ArmstrongCardaliaguet}, here we point out the critical places where the star-shaped sub-level set property needs to be used and, in these cases,  we present full proofs. 

\subsection{Outline}\label{sec: hom outline}  Before getting in the technical details we give a broad outline of the proof of quantitative homogenization from \cite{ArmstrongCardaliaguet}, which had its origin in the article \cite{ArmstrongCardaliaguetSouganidis}.

\medskip

As was previously established in \cite{ArmstrongSouganidis2,ArmstrongSouganidis}, an efficient way to  study the homogenization of \eqref{eqn: HJ1} is to look at the asymptotic behavior of  the solutions to the so-called ``metric problem" to a closed set $S \subset \real^d$. The latter is to seek, for each $\mu >0$,  a positive solution $m_\mu$ to  
\begin{equation}\label{eqn: metric to S 0}
 H(Dm_\mu,x) = \mu \ \hbox{ in } \ \real^d \setminus S \ \hbox{ with } \ m_\mu(x) \leq 0 \hbox{ on } S.
 \end{equation}
In order to make the connection between the metric  and the original homogenization problems more transparent,  we consider here an eikonal or level-set evolution type Hamiltonian of the form
\[ H(\xi,x) =  a(\tfrac{\xi}{|\xi|},x)|\xi|.\]
This choice of  $H$ may  seem to be restrictive, but, in fact, the  metric problems for star-shaped sub-level set Hamiltonians are equivalent to metric problems for level-set evolutions; see subection~\ref{sec: reduction} for more details.

\medskip

In view of the positive homogeneity of the Hamiltonian  we are using here, it turns out that it is enough to consider the 
$\mu=1$ metric problem, and to simplify the presentation, next we write $m$ instead of $m_1$.

\medskip

An intuitive way to understand the role of $m$, is to think of $m(x)$ as being  exactly the arrival time at the location $x$ of a front starting at $S$ and moving with outward normal velocity $a(n,\cdot)$.  This can be seen by looking at $u(x,t) = m(x) - t$ which evidently solves
\[ u_t + H(Du,x) = 0 \ \hbox{ with } \ \partial \{ u(x,0) = 0\} = \partial S.\]
Here we are, in particular, interested in the planar metric problem, that is \eqref{eqn: metric to S 0} with
a half plane target set $S$, that is 
\[S=\mathcal{H}_e^\pm: = \{ x \in \real^d : \pm x \cdot e \geq 0 \} \ \hbox{ for some } \ e \in S^{d-1}.\]
In this case the $t-$ level-set of the metric problem solution $m$ is the  location at time $t$ of a front starting  from the hyperplane $\{y\in\R^d: y \cdot e = 0\}$ and moving with normal velocity $a(n,\cdot)$.  
\medskip

For the purposes of homogenization it is important to understand   the long time behavior of this front.  In particular, if we can find a non-random asymptotic speed $\overline{a}(e)$ for every direction $e$, then the  homogenized Hamiltonian is $\overline{H}(\xi) = \overline{a}(\frac{\xi}{|\xi|})|\xi|$.

 \medskip
 
The uniqueness property for the metric problem, which is guaranteed by the star-shapedness condition, comes into the proof of quantitative homogenization in two important places which we discuss next in this hand waving outline. 
 
 \medskip
 
The nice feature of geometric (level-set) evolutions is that they inherently localize the spatial dependence of the solution on the coefficients of the problem.  This is an extremely useful feature in quantitative random homogenization, where one is always studying this question of how the solution depends on the coefficients of the equation.  
\medskip

The localization property of the geometric evolutions, stated informally, is that  the location of the front at time $t$ depends only on the region traced out by the front on the interval $(0,t)$ or in other words,
 \[ \hbox{$m(x)$ depends only on the values of $H$ for $x \in \{ y: m(y) \leq m(x)\}$.} \]
This is an immediate consequence of the uniqueness of the level-set evolution; see Lemma~\ref{lem: localization in sublevels}, and  a critical place where the comparison principle  for the metric problem comes into the proof of quantitative homogenization.  
 
 \medskip
 
 The property of the previous paragraph, which is referred to as localization in sub-level sets, is the key tool in obtaining a martingale decomposition of $m(x)$ -- this idea goes back to the work of Kesten \cite{Kesten} on first passage percolation.  
 \medskip
 
 Loosely speaking,  let  $\mathcal{F}_t $ be the minimal filtration  of the probability space which makes the set valued random variables $\{x: m(x) \leq t\}$ to be $\F_t$-measurable for all $t >0$. The existence of $\mathcal{F}_t $  is, of course, a bit tricky are considering continuous and not discrete in time  setting. 
\medskip  
  
The martingale we are interested in is 
 \[ \E[m(x)|\F_t] - \E[m(x)].\]
The lower bound on the front propagation speed implied by the coercivity yields  that, if $t \gtrsim x \cdot e$, then   almost surely $x \in \{y\in \R^d:m(y) \leq t\}$.  It then follows from the localization in sub-level sets property that  $m(x) \in \mathcal{F}_t$ and, hence, 
  \[ m(x) - \E[m(x)] =  \E[m(x)|\F_t] - \E[m(x)] = \sum_{k=0}^{[t]} \E[m(x)|\F_{k+1}] - \E[m(x)|\F_{k}], \]
  which is a sum of (bounded) martingale differences.  
  \medskip
  
  Then a classical concentration estimate, known as Azuma's inequality, gives the following estimate on the probability that $m(x)$ deviates too much from its mean $\E m(x)$: 
  \begin{equation}\label{eqn: fluctuations est 0}
   \P(|m(x) - \E m(x) | \geq \lambda (x \cdot e)^{1/2}) \leq C \exp(-c\lambda^2).
   \end{equation}
 Having stablished an estimate on the concentration of $m(te)$ about its mean $\E m(te)$,  we are left  with the task to understand the limiting behavior of $\E m(te)$ as $x \cdot e \to +\infty$.

\medskip

 It turns out (see Proposition~\ref{prop: bias estimate}) that $\E m(te)$ is  approximately linear in $t$ as $t \to \infty$, a fact that is expressed in the estimate
 \[ |\E m((t+s)e) - \E m(te)-\E m(se)| \leq Ct^{1/2} \log^{1/2}(1+t) \ \hbox{ for } \ 1\leq s \leq t.\]
 This approximate linearity comes from the concentration estimate combined with a semi-group (uniqueness again) property of $m$.  The idea of the proof is as follows. The concentration about the mean established in \eqref{eqn: fluctuations est 0} implies that,  when $\lambda$ is large,  $|m(te) - \E m(te)| \leq \lambda t^{1/2}$ with high probability.  In fact, using the stationarity, continuity and a union bound,  is not hard to show  that, except for an event of very small probability, $m(x+te) - \E m(x+te)$ is not too much larger than $t^{1/2}$ for any $x \in B_{Rt}(0) \cap \{ y \in \R^d: y \cdot e = 0\}$ for large $R$.  In other words, $m$ looks, with reference to the scale $t$, like the constant $\E m(te)$ on $(\partial \mathcal{H}_e^++te) \cap B_{Rt}(te)$.  Here is where the semi-group/quantitative uniqueness property comes in again, since we can now compare $m$ with the metric problem solution in $\mathcal{H}_e^++te$ with boundary data $\E m(te)$ on $\partial(\mathcal{H}_e^++te)$. Then the fluctuations estimate give that, again with high probability, $m$  will be close to $\E m(se)+\E m(te)$ at $(t+s)e$.  This is the second important way that uniqueness comes into the proof of quantitative homogenization.

 \medskip
 
 Although we have been trying to emphasize in this outline the important role played by the uniqueness property of the metric problem, we conclude our summary by explaining why uniqueness could perhaps be not as important as it seems.  If there are indeed multiple (relevant) solutions of the metric problem, it is conceivably possible that there could be some way of keeping track of each solution separately, for example, by a new metric type problem now with uniqueness.  In other words given a particular solution of the metric problem one could attempt to find, varying $\mu$, a unique continuation of that branch of solutions.  In a sense this is the approach carried out by \cite{ATY} for a particular class of equations.

\subsection{The assumptions on the Hamiltonian and the random field and some extensions} \label{sec: assumptions}
Aiming to avoid unnecessary complications  we make the following two assumptions, the first being a normalization and the second a simplification:
 \begin{equation}\label{eqn: assumption min H}
 \esssup H(0,0) = 0 \ \hbox{ and } \ \P( H(0,0) = 0) = c_0 >0. 
 \end{equation}
Given the normalization $\esssup H(0,x) =0$, it is shown in \cite{ArmstrongCardaliaguetSouganidis} that $\P(H(0,0) = 0)$ controls the lower deviations of the approximate correctors $\P(-\delta v^\delta(0,\xi) \ll -\delta)$ for all $\xi$. As a consequence,  the homogenized Hamiltonian, if it exists, satisfies $\min\overline{H} \geq 0$.  This result depends only on coercivity and no additional structural properties of the Hamiltonian.

\medskip
 
Next we make precise our assumption on the quantitative star-shapedness of the $\mu>0$ sub-level sets $\{ \xi: H(\xi,x) \leq \mu\}$.   We assume that there exists a modulus $\omega: [0,\infty) \to [0,\infty)$, which is positive for $\mu >0$ such that, for all $(\xi,x) \in \real^d \setminus \{0\} \times\real^d,$ 
 \begin{equation}\label{eqn: assumption strongest}
 \tfrac{\xi}{|\xi|} \cdot D_\xi H(\xi,x) > \omega(H(\xi,x)\vee0).
 \end{equation}
Note that \eqref{eqn: assumption strongest} indeed implies that the $\mu$ positive sub-level sets of $H$ are strictly star-shaped.  

\medskip

We also need the standard coercivity, growth and continuity bounds on the Hamiltonian, that is we assume that there are exist $0<c_0,C_0 < \infty$ and $p \geq q \geq 1$ such that, for all $\xi, x \in \R^d$,
  \begin{equation}\label{eqn: assumption bounds}
  c_0 |\xi|^q - C_0 \leq H(\xi,x) \leq C_0 |\xi|^p+C_0,
 \end{equation}
and 
 \begin{equation}\label{eqn: assumption cont}
 |D_xH(\xi,x)| + (1\vee|\xi|)|D_\xi H(\xi,x)| \leq C_0 (1 \vee |\xi|^p).
 \end{equation}

\medskip

When working with \eqref{eqn: HJ1}, to simplify statements and shorten writing we combine the above assumption on $H$ in 
\begin{equation}\label{takis100}
H \ \text{ satisfies \ \eqref{eqn: assumption strongest},  \eqref{eqn: assumption bounds},  \eqref{eqn: assumption cont} and \eqref{eqn: assumption min H} and we call } \ \textup{data} = (d,c_0,C_0,q,p).
\end{equation}
We  work on the probability space $(\Omega,  \mathcal{F}, \P)$.  Here  $\Omega$ is the set of all fields of Hamiltonians satisfying the above assumptions, that is  
 \begin{equation}\label{takis101}
 \Omega := \{ H : \real^d \times \real^d \to \real : H \ \hbox{ satisfies  \eqref{takis100}}\}.
 \end{equation}
For each $U \subseteq \real^d$ Borel we define the cylinder $\sigma$-algebra $\mathcal{F}(U)$
 \begin{equation*}
 \mathcal{F}(U) := \sigma( H \mapsto H(\xi,x) : x \in U, \ \xi \in \real^d).
 \end{equation*}
Then the $\sigma$-algebra $\mathcal{F}$ is taken to be the largest of the $\mathcal{F}(U)$,
\begin{equation}\label{takis102}
\mathcal{F} := \mathcal{F}(\real^d).
\end{equation} 
The set $\Omega$ is endowed with the group of translations $(T_y)_{y \in \R^d}$ given, for each $y \in \real^d$,  by
\[T_yH(\cdot,\cdot): = H(\cdot,\cdot + y).\]
We assume that the probability measure $\P$, which remains fixed throughout  this section, is stationary and has unit range of dependence, that is, respectively,
\begin{equation}\label{takis103}
\text{ the map $T_y$ is $\P$-preserving for every $y \in \real^d,$}
\end{equation}
and 
\begin{equation}\label{takis104}
\begin{cases}
\text{ the $\sigma$-algebras $\mathcal{F}(U)$ and $\mathcal{F}(V)$ are}\\[1mm]
\text{ $\P$-independent  for every pair of Borel sets $U,V$ in $\R^d$, with $d(U,V) \geq 1$.}
\end{cases}
\end{equation}
We summarize the properties in 
\begin{equation}\label{takis105}
H \mapsto H(\cdot,x) \ \text{ is a stationary and $1$-dependent random field of Hamiltonians on $\real^d$.}
\end{equation}

We discuss briefly  more general types of Hamiltonians where these methods could be applied.  The Hamiltonian 
 $$ H(\xi,x) = |\xi|^2-b(x) \cdot \xi,$$
has star-shaped level sets of $H$ in a uniform way, but the center of the star-shapedness is varying in $x$.  Indeed,
 $$H(\xi,x) =  |\xi -b(x)|^2 - \frac{1}{4}|b(x)|^2,$$
 and it is clear that the sub-level sets of $H(\cdot,x)$ are star-shaped with respect to $b(x)$.  
 \medskip
 
 If $b(x) = D B(x)$, where $B$ is a bounded $C^1-$ vector field,
 the transformation to,
 $$v(x) = u(x) - B(x),$$
yields that,  if $u$ solves $H(Du,x) = \mu$ in $U\subset R^d$,   then $v$ solves 
 $$ |Dv|^2 - \frac{1}{4}|b(x)|^2 = \mu,$$
which is an equation satisfying the assumptions of the previous section.  
\medskip

The observation above shows that we can work with Hamiltonians $H$, such that, that there exists  a bounded stationary random potential field $B \in C^{1}(\real^d;\R^d)$ with finite range of dependence such that  \eqref{takis100} is satisfied by the Hamiltonian 
 $$G(Du,x) := H(Du + DB,x).$$
 
 \medskip
 
 In a different direction,  the individual sub-level sets of $H(\cdot,x)$ for fixed $x$ can be uniformly star-shaped with respect to different points.  Indeed we can assume that, for every $\mu>0$, there exists $\xi_\mu$ such that, for every $x\in \real^d$, $\{\xi: H(\xi,x)  \leq \mu\}$ is star-shaped with respect to $\xi_\mu$ and
  \begin{equation}\label{eqn: assumption strongest general}
 \tfrac{\xi-\xi_\mu}{|\xi-\xi_\mu|} \cdot D_\xi H(\xi,x) \geq \omega(\mu) \ \hbox{ for all } \ x\in \real^d, \ \xi \in \{ H(\xi,x) = \mu\}.
 \end{equation}
 Quite similar estimates to those we derive below will hold in this setting as well.  
 
 \medskip
 
 Several other generalizations are evident, including a combination of the two we have mentioned.  
\medskip
 
Another possibility, out of reach of our current methods, is to consider Hamiltonians with the property  that every connected component of every level set of $H$ is the boundary of a star-shaped set.  Formally one can separate out each connected component of each level set getting an eikonal type equation for each, but the connection with the original equation Lemma~\ref{lem: eikonal eqn} is lost.

 \medskip
 
 We conclude the discussion about the general assumptions commenting that the main property  needed to study the ``viscous'' Hamilton-Jacobi problem \eqref{eqn: viscous HJ2} turns out to be something between star-shaped sub-levels and homogeneity.  
 
 \medskip
 Indeed the conditions is that, for every $\mu>0$, there exists $\xi_\mu \in \R^d$ such that, for every $x\in \real^d$, the level set $\{\xi \in \R^d: H(\xi,x)  \leq \mu\}$ is star-shaped with respect to $\xi_\mu$ and,  
 \begin{equation}\label{eqn: assumption viscous ss}
 \tfrac{\xi - \xi_\mu}{|\xi-\xi_\mu|} \cdot D_\xi H(\xi,x) \geq \mu \vee 0 \ \hbox{ for all } \ x\in \real^d \  \text{and} \  \xi \in \{\xi'\in \R^d: H(\xi',x) = \mu\}.
 \end{equation}
Hamiltonians which are $p$-homogeneous $H$ with  $p \geq 1$ satisfy  \eqref{eqn: assumption viscous ss}. However,  \eqref{eqn: assumption viscous ss} allows for more general $H$'s like, for example, logarithmic terms $H(\xi) = |\xi| \log(1+|\xi|)$ or direction dependent homogeneity as, for example,  $H(\xi) = |\xi|^{p(\frac{\xi}{|\xi|})}$ with $p :S^{d-1} \to [1,\infty)$.

 \subsection{Reduction to an eikonal/level-set equation and the radius function}\label{sec: reduction}
In subsection~\ref{sec: hom outline} we explained in a heuristic way the homogenization proof for level-set type Hamiltonians.  Here we show that metric-type problems for sub-level star-shaped Hamiltonians can be transformed into metric-type problems for a $1$-homogeneous level-set type Hamiltonian.  As a result  the methods explained in Section~\ref{sec: hom outline} can also be applied to the class of sub-level star-shaped Hamiltonians.

 \medskip
 
Assume that \eqref{eqn: intro star-shapedness} holds, that is, for  every $\mu>0$, the  sub-level $\{ H(\xi,x) \leq \mu\}$ is strictly star-shaped with respect to $0$.  Then  there exists a ``radius''-function  $r_\mu:S^{d-1}\times \R^d \to (0,\infty)$ such that,
 \begin{equation}\label{eqn: rmu} 
 H(r \hat\xi,x) < \mu \ \hbox{ for } \ r < r_\mu(\hat\xi) \ \hbox{ and } \ H(r \hat\xi,x) > \mu \ \hbox{ for } \ r > r_\mu(\hat\xi).
 \end{equation}
It follows that the metric problem for $H$ is equivalent to the one for $\tilde H(\xi, x):=r_\mu(\hat\xi, x)^{-1}|\xi|.$
\begin{lem}\label{lem: eikonal eqn}
Let $U$ be a domain of $\real^d$ and assume that \eqref{eqn: intro star-shapedness} holds. Then 
\[r_\mu(\tfrac{Du}{|Du|},x)^{-1}|Du| = 1 \ \hbox{ in } \ U \ \hbox{ if and only if } \ H(Du,x) = \mu \ \hbox{ in } \ U.\]
\end{lem}
The proof is an almost immediate consequence of the definition of viscosity solutions.
\begin{proof}
If for some $x\in U$, $\xi\in D^+u(x,t)$,  the sub-solution property, yields
\[r_\mu(\hat\xi,x)^{-1}|\xi| \leq 1,\]
and the definition of $r_\mu$ in \eqref{eqn: rmu} gives, 
\[H(\xi,x) \leq \mu.\]
that is $u$ is a sub-solution of $H(Du,x) \leq \mu$ in $U$.  A similar argument works for the super-solution property as well as the other direction in the claim,
\end{proof}

Given  \eqref{takis100}, we show that the radius function $r_\mu$ satisfies all the properties needed for homogenizing the level-set equation $|Dm_\mu| = r_\mu(\frac{Dm_\mu}{|Dm_\mu|},x)$. These are  upper and lower bounds, continuity in $e$, and monotonicity and continuity in $\mu$.

\medskip

\begin{lem}\label{lem: r lip} 
(i)~There exist constants $C,c>0$ such that, for all  $e \in S^{d-1}$ and $x\in\R^d$,  
$$0< a_\mu< \mu(e,x) < A_\mu,$$
with  $a_\mu = c (\mu^{1/p} \wedge \mu)$ and  $A_\mu = C(\mu + C)^{1/q}$.

\medskip

(ii)~For every $\mu>0$, $e \in S^{d-1}$ and $x \in \real^d$, 
\[ C_0^{-1}A_\mu^{1-p}\leq \frac{d}{d\mu}r_\mu(e,x) \leq \frac{1}{\omega(\mu)}.\]
Furthermore $e \to r_\mu(e,x)$ is Lipschitz continuous on $S^{d-1}$ with Lipschitz constant $L_\mu: =  C_0a_\mu^{-1}\omega(\mu)^{-1}A_\mu^p.$
\end{lem}

\begin{proof} It follows from  \eqref{eqn: assumption bounds} that 
$$ \mu= H(r_\mu(e,x)e,x) \geq c_0 r_\mu(e,x)^{q}-C_0, $$
while \eqref{eqn: assumption bounds}, \eqref{eqn: assumption cont} and  \eqref{eqn: assumption min H} give 
$$ \mu = H(r_\mu(e,x)e,x) \leq C_0(r_\mu(e,x)+p^{-1}r_\mu(e,x)^p).$$
Then
\[ c(\mu^{1/p} \wedge \mu) \leq r_\mu(e,x) \leq C(\mu +C)^{1/q}\]
and the claim follows with  $A_\mu = C(\mu + C)^{1/q}$ and $a_\mu = c (\mu^{1/p} \wedge \mu)$.

\medskip

Next we examine the dependence of $r_\mu$ on $\mu$.  The Lipschitz estimate in \eqref{eqn: assumption bounds} gives 
\begin{equation}\label{eqn: monotonicity 1}
\mu = H(r_\mu(e,x)e,x) \leq H(r_\nu(e,x)e,x) + C_0A_\mu^{p-1}(r_\mu(e,x)-r_\nu(e,x)),
\end{equation}
and, using \eqref{eqn: assumption strongest},  for any $\mu > \nu >0$, we obtain
$$ \mu = H(r_\mu(e,x)e,x) \geq H(r_\nu(e,x)e,x) +\omega(\mu)\wedge\omega(\nu)(r_\mu(e,x)-r_\nu(e,x)).$$
Rearranging the above inequalities,  using that $H(r_\nu(e,x)e,x) = \nu$ and combining with \eqref{eqn: monotonicity 1} yields the continuity estimate
\begin{equation}\label{eqn: r mu cont}
C_0^{-1}A_\mu^{1-p}(\mu - \nu) \leq (r_\mu(e,x)-r_\nu(e,x)) \leq \frac{1}{\omega(\mu)\wedge\omega(\nu)}(\mu - \nu).
\end{equation}
Sending $\nu \to \mu$ and/or $\mu \to \nu$ gives the result.
 
\medskip
 
Finally, we consider the Lipschitz estimate of $r_\mu(\cdot,x)$ on the unit sphere.  Fix $e \in S^{d-1}$ and  $\delta >0$, let $e' \in S^{d-1}$ be such that $|e' - e| = \delta$, define
 $\lambda$ by solving the relation 
 \[\lambda =1+2C_0\omega(\mu)^{-1}\lambda^{p-1} r_\mu(e,x)^{p-1}\delta,\]
  and note that $\lambda \to 1$ as $\delta \to 0$.  In particular we can take $\delta$ sufficiently small so that $2>\lambda >1$.  
 
 \medskip
 
 The continuity of  $\omega$ and $H$ yield that, for $\delta$ small and $\lambda \approx 1$, $\omega(H(\lambda r_\mu(e,x)e,x))\geq  \tfrac{1}{2} \omega(\mu)$ and, hence, in view of  \eqref{eqn: assumption strongest}, 
\[ H(\lambda r_\mu(e,x)e,x) \geq  H(r_\mu(e,x)e,x)+(\lambda-1) \omega(H(\lambda r_\mu(e,x)e,x)) \wedge \omega(\mu) = \mu + (\lambda-1) \tfrac{1}{2} \omega(\mu),\]
and, in view of \eqref{eqn: assumption cont},
\[ H(\lambda r_\mu(e,x)e',x) \geq \mu + (\lambda-1) \tfrac{1}{2} \omega(\mu) - C_0(\lambda r_\mu(e,x))^{p-1}  \delta.\]
By the  choice of $\lambda$,  we have $H(\lambda r_\mu(e,x)e',x) \geq \mu$ which, by the definition of $r_\mu(e',x)$, gives 
\[ r_\mu(e',x) \leq \lambda r_\mu(e,x)  \leq  r_\mu(e,x) + 2\omega(\mu)^{-1}\lambda^{p-1}C_0r_\mu(e,x)^p\delta.\]
Sending $e' \to e$ so that $\delta = |e-e'| \to  0$ and $\lambda \to 1$,  we obtain the desired result
\[ |D_er_\mu(e,x)| \leq 2C_0\omega(\mu)^{-1}r_\mu(e,x)^p. \]
\end{proof}

\subsection{The planar metric problem}\label{sec:planar}
We establish some basic properties of the metric or minimal time problem for star-shaped sub-level Hamiltonians.  Given a closed set $S \subset \real^d$ and $\mu >0$ we look for a non-negative solution $m_\mu(x,S,H)$ to the problem,
\begin{equation}\label{eqn: metric to S}
H(Dm_\mu,x) = \mu \ \hbox{ in } \real^d \setminus S \ \text{and} \ m_\mu = 0 \  \hbox{ on }  \partial S.
\end{equation}
As described in subsection~\ref{sec: hom outline}, of particular interest is the planar metric problem, that is \eqref{eqn: metric to S} with a half-space target set $S=\mathcal{H}_e^\pm: = \{ x \in \real^d : \pm x \cdot e \geq 0 \}$, for some unit direction  \ $e \in S^{d-1}$.  A possible solution to \eqref{eqn: metric to S} is the maximal sub-solution. There may exist, however, more solutions -- for example this is the case when $H$ is spatially homogeneous and has a sub-level which is not star-shaped.  

\medskip

We show that, for star-shaped sub-level Hamiltonians, the planar metric problem \eqref{eqn: metric to S} has a comparison principle and, hence, a unique solution.  This is the key reason that allows the use of the planar metric problems to prove homogenization.

\medskip

In the first-order case we are considering here, we can use the transformation of subsection~\ref{sec: reduction} to show uniqueness. In fact, uniqueness is a corollary of the uniqueness of metric problems for homogeneous Hamiltonians established, for example, in \cite{ArmstrongCardaliaguet}.  Indeed,  Lemma~\ref{lem: eikonal eqn}  says that  any solution to \eqref{eqn: metric to S} by Lemma~\ref{lem: eikonal eqn} is also a solution for the Hamiltonian $r_\mu^{-1}(\hat\xi,x)|\xi|$. Since comparison/uniqueness holds for the latter problem, the Hamiltonian being positively homogeneous, comparison/uniqueness holds for $H$ as well.  This is in essence 
the proof of the next lemma which is omitted.

\begin{lem}\label{lem: pmp comp}
Assume \eqref{takis100} and let  $m^1$ and $m^2$  be respectively a subsolution and a non-negative supersolution of \eqref{eqn: metric to S}.  Then $m^1 \leq m^2$ in $\real^d \setminus S$. 
\end{lem}

%

\subsection{The fluctuations estimate}  \label{sec: The fluctuations estimate}
We consider the stochastic part of the error estimate for the metric problem solutions, that is the fluctuations around the mean
\[ m_\mu(x,S) - \E m_\mu(x,S).\]
We state without details an estimate of the stochastic fluctuations for $S \subseteq \real^d$ compact.  The claim can be derived as a corollary of Proposition 3.1 of\cite{ArmstrongCardaliaguet} using the transformation presented previously. We remark that the dependence of the error on $\mu$ is better than the one obtained in \cite{ArmstrongCardaliaguet} because here  we consider only the first-order problem. The calculation of the exact constant is tedious and requires a careful reading through the proof of \cite{ArmstrongCardaliaguet}. 
\begin{prop}[Fluctuations estimate]\label{prop: fluctuation est}
Assume \eqref{takis100} and \eqref{takis105} and let 
 $S \subseteq \real^d$ be a compact set.  Then, for every $\overline{\mu}>1$,  there exists $C=C(d,\overline{\mu})\geq 1$ such that, for all $\mu \in (0,\overline{\mu}]$ and  $x\in \real^d \setminus S,$
\begin{equation}
\P(|m_\mu(x,S) - \E[ m_\mu(x,S)] |\geq \lambda) \leq C\exp\left( - \frac{a_\mu\lambda^2}{C(1+d(x,S))}\right).
\end{equation}
\end{prop}
As described in subsection~\ref{sec: hom outline}, Proposition~\ref{prop: fluctuation est} is proved by a decomposition of the difference $m_\mu(x,S) - \E[ m_\mu(x,S)]$ as a sum of bounded martingale differences.  The key lemma (Lemma~\ref{sec: hom outline}) needed for this martingale decomposition is the localization in sub-level sets property of geometric evolutions. Informally speaking, this says that $m_\mu(\cdot,S,H)$ in the sub-level $\{ m_\mu \leq t\}$ depends only on the values of $H$ there.  Its  proof is based on the comparison principle for the metric problem and, hence, it relies on the star-shapedness of the sub-levels. 
 \begin{lem}
 \label{lem: localization in sublevels}
Assume \eqref{takis100}, \eqref{takis105} and fix  $H_1, H_2  \in \Omega$ and $t >0$. If $ H_1 \equiv H_2 \ \hbox{ in } \ \real^d \times \{ x \in \real^d \setminus S: m_\mu(x,S,H_1) \leq t\}$, then   
 \[ m_\mu (\cdot,S,H_1) \leq m_\mu(\cdot,S,H_2) \ \text{ in} \ \{x \in \real^d \setminus S:m_\mu(x,S,H_1) \leq t\} .\]
 \end{lem}
 \begin{proof}
Let  $m_i := m_\mu(\cdot,S,H_i)$ with $i=1,2$, set  $R_t: = \{ x \in \real^d \setminus S: m_\mu(x,S,H_1) \leq t\}$.  fix   a smooth $\phi : [0,\infty) \to [0,t]$ such that $\phi' \leq 1$ and $\phi(s) = t$ for all $s \geq t$. 

\medskip

It follows that $w: = \phi(m_1)$ is a subsolution of $H_2(Dw,x) \leq \mu$ in $\real^d \setminus S$.  Indeed this is the case  in $\real^d \setminus R_t$ where  $w$ is constant there and the assumptions on $H$ yield  that constant functions are subsolutions $H(0,x) \leq 0 \leq \mu$ of $H \in\Omega$.   Using that $H_1 \equiv H_2$ in $R_t$, $\phi' \leq 1$ and the star-shapedness condition \eqref{eqn: assumption strongest}, we find 
  \[H_2(Dw,x) = H_2(\phi'(m_1(x))Dm_1,x) = H_1(\phi'(m_1(x))Dm_1,x) \leq \mu \ \hbox{ in } \ R_t.\]
Now the comparison principle for the planar metric problem associated with the operator $H_2$ that is  Lemma~\ref{lem: pmp comp} implies that $w \leq m_2$.  Finally,  letting  $\phi(s)$ to converge to $s \wedge t$, and, hence, $w \to m_1 \wedge t$ yields that $m_1 \leq m_2$ in the set $R_t$.
\end{proof}

\subsection{The bias estimate}  \label{sec: bias estimate}
Next we discuss the deterministic part of the error estimate, the convergence of the expected values $\E[m_\mu(x, \mathcal{H}_e^-)]$. Again our result is a corollary of \cite{ArmstrongCardaliaguet} by the level-set transformation.   
\begin{prop}\label{prop: bias estimate}
Let $\bar\mu \geq 1$, $\mu \in (0,\bar\mu]$ and $e \in S^{d-1}$.  There exists $\overline{r}_\mu(e)>0$ and $C=C(\textup{data},\bar\mu)\geq1$ such that, for every $x \in \mathcal{H}_e^+$,
\[ |\E[m_\mu(x, \mathcal{H}_e^-)] -\overline{r}_\mu(e) (x \cdot e)| \leq C_\mu(x \cdot e)^{1/2}\log^{1/2}(1+(x \cdot e)),\]
where $C_\mu = C(\textup{data},\overline{\mu}) a_\mu^{-1/2}(1+|\log L_\mu|)^{1/2}$.  Moreover, $(\mu,e) \to \overline{r}_\mu(e)$ is continuous on $(0,\infty) \times S^{d-1}$. 
\end{prop}
The proof of Proposition~\ref{prop: bias estimate} follows, as in \cite{ArmstrongCardaliaguet}, by a series of lemmata that are stated below without a proof.   The comparison principle for the metric problem again plays a key role. 
  
\medskip

The objective is  to show that $\E m_\mu(te,\mathcal{H}_e^-)$ is almost additive in $xt$. 
\begin{lem}\label{lem: almost additivity}
For all $\mu \in (0,\overline{\mu}]$ and  $t\geq s \geq 1$,
$$|\E m_\mu((t+s)e,\mathcal{H}_e^-) - \E m_\mu(te,\mathcal{H}_e^-) - \E m_\mu(se,\mathcal{H}_e^-)| \leq C_\mu(s\vee t)^{1/2}\log^{1/2}(1+s\vee t),$$
where $C_\mu = C(\overline{\mu}) a_\mu^{-1/2}(1+|\log L_\mu|)^{1/2}$.
\end{lem}

The fluctuation estimate guarantees that, with high probability, $m_\mu$ is close to its mean in a large ball $B_{Rt}(te)$. Define
\begin{equation}
N_R^+(t) := \sup_{x \in B_{Rt}(te) \cap (\partial \mathcal{H}_e^+ + te)} m_\mu(x,\mathcal{H}_e^-) \ \ \text{and} \  \ N_R^-(t) := \inf_{x \in B_{Rt} \cap (\partial \mathcal{H}_e^+ + te)} m_\mu(x,\mathcal{H}_e^-).
\end{equation}
\begin{lem}\label{lem: union bound}
There exists $C=C(\overline{\mu})>0$ such that, for all  $\mu \in (0,\overline{\mu}]$ and $R, t >1$,
\begin{equation}\label{eqn: npm}
 \E|N^\pm_R(t) - \E[m_\mu(te,\mathcal{H}_e^-)]| \leq C a_\mu^{-1/2}t^{1/2}\log^{1/2}(1+Rt).
 \end{equation}
\end{lem}

Since $m_\mu(\cdot,\mathcal{H}_e^-)$ is approximately a constant on $B_{Rt}(te) \cap \{ x\in \R^d: x \cdot e = t\}$  we compare $m_\mu$ with $m_\mu(\cdot,\mathcal{H}_e^-+te)+\E m_\mu(te,\mathcal{H}_e^-)$ in the domain $\mathcal{H}_e^++te$.  This requires a quantitative and localized version of uniqueness for the planar metric problem.  Then the error in linearity of $\E m_\mu(te,\mathcal{H}_e^-)$ can be rewritten as,
\[   |\E (\E m_\mu(te,\mathcal{H}_e^-)+m_\mu(se,\mathcal{H}_e^-+te) -  m_\mu((t+s)e,\mathcal{H}_e^-))|\]
where the quantity inside the expectation will be estimated by the comparison/uniqueness result.

\medskip

For the argument we need the following refinement  of the metric problem uniqueness, which  is very close in spirit to finite speed of propagation.
\begin{lem}\label{lem: finite speed}
Fix $e \in S^{d-1}$. Let $m^1\in C(\mathcal{H}_{e}^+)$ and  $m^2 \in  C(\mathcal{H}_{e}^+)$ be respectively a  sub-solution and a locally Lipschitz  super-solution to $ H(Dm,y) = \mu \ \hbox{ in } \ \mathcal{H}_e^+$. Assume that there exist $M,K >0$ and $ R>L_\mu M$, $L_\mu$ being the upper bound for the Lipschitz constant of $r_\mu^{-1}( \cdot,x)$ in Lemma~ \ref{lem: r lip},  such that,  
\[ 0 \leq m^1\leq M + K|x|, \ \ 0 \leq m^2 \leq M + K|x| \ \text{ and} \  m_1 \leq m_2 \ \hbox{ on } \ \partial \mathcal{H}_e^+ \cap B_R.\]
 Then, for $s>0$ such that $R \geq L_\mu M + (K+1)s$, 
\[ m_1(se) \leq m_2(se).\]
\end{lem}
Now from Lemma~\ref{lem: union bound} and Lemma~\ref{lem: finite speed} together we can prove the almost additivity (Lemma~\ref{lem: almost additivity}).
\begin{proof}[Proof of Lemma~\ref{lem: almost additivity}.]   Without loss we may assume  that $ s \leq t$. Let $N^\pm$ as defined above in \eqref{eqn: npm} with $R = L_\mu A_{\bar\mu}+(L_\mu+A_{\bar\mu})t $.  
\medskip

We apply  Lemma~\ref{lem: finite speed} to $m^1(x) = m_\mu(x,\mathcal{H}_e^-)$ and $m^2(x) = m_\mu(x,\mathcal{H}_e^++te)+N^+(t)$ in the domain $\mathcal{H}_e^++te$.  Checking the hypotheses of the localization result, we see that $m^1$ and  $m^2$ are both solutions to
$$ H(Dm^j,x) = \mu \ \hbox{ in } \ \mathcal{H}_e^+ + te.$$
Recalling that $|Dm| = r_\mu(\tfrac{Dm}{|Dm|},x) \leq A_\mu$, we obtain the Lipschitz estimate $\|Dm^j\|_{\infty} \leq A_\mu$ and,  in particular,
\[ 0 \leq m^j(x) \leq A_\mu (t+|x-te|) \ \hbox{ in } \ \mathcal{H}^+_e.\]
and then we  take $M=K= A_\mu \leq A_{\overline{\mu}}$ in the statement of Lemma~\ref{lem: finite speed}.  

\medskip
Finally, since  the ordering $m^1 \leq m^2$ on $B_{R}(te) \cap (\partial \mathcal{H}^+_e+te)$ follows directly from the definition of $N^+(t),$  we find using the particular our choice of $R$ above,
\[ m_\mu((t+s)e,\mathcal{H}^-_e) \leq  m_\mu(se,\mathcal{H}^-_e+te)+N^+(t).\]
Taking expectations on both sides and using the $\real^d$-stationarity of $m_\mu(x,\mathcal{H}_e^++\cdot)$,
we obtain 
\begin{align*}
\E m_\mu((t+s)e,\mathcal{H}^-_e) &\leq  \E m_\mu(se,\mathcal{H}^-_e+te)+ \E m_\mu(te,\mathcal{H}^-_e) +C(\bar\mu) a_\mu^{-1/2}t^{1/2}\log^{1/2}(1+Rt) \\
&=\E m_\mu(se,\mathcal{H}^-_e)+ \E m_\mu(te,\mathcal{H}^-_e) +C(\bar\mu) a_\mu^{-1/2}t^{1/2}\log^{1/2}(1+Rt) \\
& \leq \E m_\mu(se,\mathcal{H}^-_e)+ \E m_\mu(te,\mathcal{H}^-_e) +C(\bar\mu) a_\mu^{-1/2}(1+|\log L_\mu|)^{1/2}t^{1/2}\log^{1/2}(1+t).
\end{align*}
For the reverse inequality we compare $m^1(x) = m_\mu(x,\mathcal{H}_e^++te)+N^-(t)$ with $m^2(x) = m_\mu(x,\mathcal{H}_e^-)$.
\end{proof}
We now return to the proof of Lemma~\ref{lem: finite speed}.
\begin{proof}[Proof of Lemma~\ref{lem: finite speed}]
The  $t$-level sets of the solution $m_\mu$ to  the metric problem can be thought  as the location at time $t$ of a front  moving  with normal velocity $r_\mu^{-1}$.  Indeed,  we show that  ${w}(x,t): = m^1(x) \wedge t$ is a subsolution to 
\begin{equation}\label{eqn: t dep}
 {w}_t +r_\mu(\tfrac{D{w}}{|D{w}|},x)^{-1}|D{w}| \leq 1  \ \hbox{ in } \ \mathcal{H}_e^+ \times (0,\infty).
 \end{equation}
Assuming for the moment \eqref{eqn: t dep} and noting that the assumptions yield that 
\begin{equation}
{w}(\cdot ,0) \leq 0 \ \hbox{ in } \ \mathcal{H}_e^+ \ \hbox{ and } \ {w} \leq m^1 \leq m^2 \ \hbox{ on } \ (\partial \mathcal{H}_e^+ \cap B_R) \times (0,\infty),
\end{equation}
we make use
of the finite speed of propagation property of the time dependent Hamilton-Jacobi equation with $L_\mu$ as in the statement  to get 
\[ w(\cdot,t) \leq m^2 \ \hbox{ in } \ B_{R- L_\mu t} \cap \mathcal{H}_e^+.\]
Note that if $ t \geq m^1(se)$, which follows from the assumption of the Lemma if $t = M+Ks$, then $m^1(se) = w(se,t)$, and  we obtain that, as long as $R - L_\mu(M+Ks)\geq s,$
\[ m^1(se) \leq  m^2(se).\]
It remains to prove \eqref{eqn: t dep}, which follows by showing that $w$ is the  uniform limit, as $\lambda \to \infty$, of the $w_\lambda$'s given by 
\begin{equation}
 {w}^\lambda(x,t): = -\frac{1}{\lambda}\log(\exp(-\lambda m^1(x)) + \exp (-\lambda t)),
  \end{equation}
 which are themselves subsolutions of  \eqref{eqn: t dep}.

\medskip

Since the limit part of the assertion above is obvious, here we only check the subsolution property.  For this we argue as if the $w_\lambda$'s were smooth, the rigorous argument following from classical viscosity solutions considerations. 

\medskip

Let 
$Z(x,t): = \exp(-\lambda m^1(x)) + \exp (-\lambda t).$
It follows that 
$$ Dw^\lambda = Dm^1 Z^{-1}\exp(-\lambda m^1) \ \hbox{ and } \ w^\lambda_t = Z^{-1}\exp(-\lambda t).$$
Then, using $1$-homogeneity of the Hamiltonian and the subsolution property of $m^1$, we get
\begin{align*}
 w^\lambda_t + r_\mu(\tfrac{Dw^\lambda}{|Dw^\lambda|},x)^{-1}|Dw^\lambda| 
 &= Z^{-1}\exp(-\lambda t) + r_\mu(\tfrac{Dm^1}{|Dm^1|},x)^{-1}|Dm^1| Z^{-1}\exp(-\lambda m^1) \\
 &\leq Z^{-1}(\exp(-\lambda t)+\exp(-\lambda m^1)) = 1.
 \end{align*}
\end{proof}

Finally one obtains the rate of convergence of the expectations by using the almost additivity at a sequence of dyadic scales.  This yields the first part of the Proposition~\ref{prop: bias estimate}.
\begin{lem}\label{lem: expectations converge}
There exists $\overline{r}_\mu(e)\geq 0$ such  that, all $\overline{\mu}>1$,  $\mu \in (0,\overline{\mu}]$ and $t >1$,
$$|\tfrac{1}{t}\E m_\mu(te) - \overline{r}_\mu(e)| \leq C_\mu t^{-1/2}\log^{1/2}(1+t),$$
where $C_\mu: = C(\overline{\mu}) a_\mu^{-1/2}(1+|\log L_\mu|)^{1/2}$.
\end{lem}

\subsection{The effective Hamiltonian $\overline H$}\label{sec:effective}  Using $\overline{r}_\mu$ we define, for all  $\xi \in \real^d \setminus \{0\}$, the effective Hamiltonian $\overline {H}(\xi)$ as 
\begin{equation}\label{takis140}
\overline{H}(\xi) := \inf \{ \mu >0 : \overline{r}_\mu(\tfrac{\xi}{|\xi|}) \geq |\xi| \}.
\end{equation}

We discuss the structural properties of $\overline{r}_\mu$ and $\overline{H}$ which are inherited from the original Hamiltonian.  
\medskip

The continuity and monotonicity properties of $\overline{r}_\mu$ follow from the next lemma, which we state without proof.
\begin{lem}\label{lem: hom r prop}
For any $\overline{\mu}\geq 1$ there exists $C(\overline{\mu}) \geq 1$ such  that, $\mu \in (0,\overline \mu]$ and $e, e_1,e_2\in S^{d-1}$,
\[C(\overline{\mu})^{-1}\mu\leq \frac{d}{d\mu}\overline r_\mu(e) \leq C(\overline{\mu})\frac{1}{\mu\omega(\mu)} \ \text{and } \  |\overline r_\mu(e_1) - \overline r_\mu(e_2) | \leq C(\overline{\mu})(1+L_\mu) |e_1 - e_2|. \]
\end{lem}

\medskip
 
In the next lemma, we assert that the quantitative star-shapedness  \eqref{eqn: assumption strongest} of the sub-level sets of $H$ are inherited by the $\overline H$. Since the latter may have flat parts, which are the the zero level set $\{\xi: \overline{H}(\xi) = 0\}$,  we no longer have strict monotonicity along radii.  However, as soon as $\overline{H}(te)$ is positive, the strict monotonicity returns.  The continuity properties of $\overline{H}$ are more easily understood via the approximate cell problem which homogenizes. Since this is standard, we omit the statements.

\begin{lem}\label{lem: hom op prop}
There exists $c=c(\overline{\mu})>0$ such that, for all $\xi \in \{\xi \in \real^d: 0<\overline{H}(\xi) \leq \overline{\mu}\},$
\[\tfrac{\xi}{|\xi|}  \cdot D_\xi\overline{H}(\xi) \geq c(\overline{\mu})\overline{H}(\xi)\omega(\overline{H}(\xi)).\]
In particular, the sub-levels of $\overline{H}$ are star-shaped with respect to the origin.
\end{lem}

\begin{proof}
Recall that, if  $\mu = \overline{H}(te) >0$, then $\overline{r}_\mu(e) = t$. For any $\nu >\mu$ let $\delta: = C\nu^{-1}(\omega(\nu) \wedge \omega(\mu))^{-1}(\nu - \mu)$. It follows from the Lipschitz property of $\mu \to \overline{r}_{\mu}(e)$ that  
$$\overline{r}_\nu(e) \leq \overline{r}_\mu(e)+ \delta,$$
and, in view of  the definition of $\overline{H}$, 
$$\nu \leq \overline{H}((t+\delta)e).$$
Thus
$$ \frac{\overline{H}((t+\delta)e) - \overline{H}(te)}{\delta} \geq \frac{\nu - \mu}{\delta} = c(\overline{\mu})\nu ( \omega(\nu) \wedge \omega(\mu)).$$
Taking the limit $\nu \to \mu$ and using the continuity of $\omega$ implies the desired lower bound.
\end{proof}

\subsection{Connection with the approximate correctors}\label{sec:connections}  Finally we describe  the connection between the planar metric problem and the solution $v^\delta(x,\xi)$ of the approximate corrector problem,
$$ \delta v^\delta + H(\xi + Dv^\delta,x) = 0 \ \hbox{ in } \ \real^d.$$
The aim is to show the convergence of $-\delta v^\delta(0,\xi) \to \overline{H}(\xi)$ stated in the next proposition.
Since the star-shapedness property and uniqueness of the metric problem do not play any significant role here, we omit any discussion about proofs for which we refer to \cite{ArmstrongCardaliaguet}.
\begin{prop}\label{prop: a correctors full rate}
There exist  $C=C(\overline \mu), c=c(\overline \mu)>0$ depending on the constants from \eqref{eqn: assumption bounds}, \eqref{eqn: assumption min H} and \eqref{eqn: assumption cont} such that, for all $\xi\in \R^d$ such that $0 \leq \overline{H}(\xi)  < \overline{\mu}$  and all $\lambda,\delta >0$,
$$\P(|\delta v^\delta(0,\xi)+\overline{H}(\xi)| >  \lambda \omega'(\delta|\log\delta|) ) \leq C\exp(-c\lambda^{4\wedge d})$$
where the modulus $\omega'$ is the inverse of the modulus $\lambda \mapsto \lambda^4/|C+\log (\lambda\omega(\lambda/2))|$ defined on $[0,1]$.
\end{prop}

\bibliographystyle{plain}
\bibliography{nonconvex_articles.bib}
\end{document}